\documentclass[12pt]{amsart} 
\usepackage[mathscr]{eucal} 
\usepackage{amsmath,amsfonts} 
\parskip=\smallskipamount 
\newtheorem{theorem}{Theorem}[section] 
\newtheorem{proposition}[theorem]{Proposition} 
 
\newtheorem{lemma}[theorem]{Lemma} 
\theoremstyle{definition}

\newtheorem{remark}[theorem]{Remark} 
\newtheorem{remarks}[theorem]{Remarks}
\makeatletter 
\@addtoreset{equation}{section} 
\makeatother

\newcommand{\CC}{{\mathbb C}} 
\newcommand{\NN}{{\mathbb N}} 
 
\newcommand{\ZZ}{{\mathbb Z}}

\newcommand{\cB}{{\mathcal B}}

\newcommand{\cE}{{\mathcal E}} 
\newcommand{\cF}{{\mathcal F}} 
 
\newcommand{\cH}{{\mathcal H}} 
\newcommand{\cI}{{\mathcal I}}
\newcommand{\cJ}{{\mathcal J}} 
\newcommand{\cK}{{\mathcal K}} 
\newcommand{\cL}{{\mathcal L}} 
\newcommand{\cM}{{\mathcal M}}

\newcommand{\cR}{{\mathcal R}}

\newcommand{\cX}{{\mathcal X}}

\newcommand{\fk}{\mathbf{k}}
\newcommand{\fl}{\mathbf{l}}

\newcommand{\Ra}{\Rightarrow}

\newcommand{\ra}{\rightarrow} 
\newcommand{\ol}{\overline}

\let\phi=\varphi 
\newcommand{\iac}{\mathrm{i}}

\newcommand{\lin}{\operatorname{Lin}}
\newcommand{\supp}{\operatorname{supp}}
\newcommand{\tl}{\tilde}

\newcommand{\nr}[1]{\vspace{0.1ex}\noindent\hspace*{12mm}\llap{\textup{(#1)}}} 
 
\begin{document} 
\title[Invariant Weakly Positive Semidefinite Kernels]{Invariant Weakly Positive Semidefinite 
Kernels with Values in Topologically Ordered $*$-Spaces}\thanks{The 
second named author's work supported by a grant of the Romanian 
National Authority for Scientific Research, CNCS л UEFISCDI, project number
PN-II-ID-PCE-2011-3-0119.}
 
 \date{\today}
 
 \author[S. Ay]{Serdar Ay}
 \address{Department of Mathematics, Bilkent University, 06800 Bilkent, Ankara, 
Turkey}
 \email{serdar@fen.bilkent.edu.tr}
 
\author[A. Gheondea]{Aurelian Gheondea} 
\address{Department of Mathematics, Bilkent University, 06800 Bilkent, Ankara, 
Turkey, \emph{and} Institutul de Matematic\u a al Academiei Rom\^ane, C.P.\ 
1-764, 014700 Bucure\c sti, Rom\^ania} 
\email{aurelian@fen.bilkent.edu.tr \textrm{and} A.Gheondea@imar.ro} 

\begin{abstract} We consider weakly positive semidefinite 
kernels valued in ordered $*$-spaces with or without 
certain topological properties, and investigate their 
linearisations (Kolmogorov decompositions) as well as their reproducing kernel
spaces. The spaces of realisations are of VE (Vector Euclidean) or VH 
(Vector Hilbert) type, more precisely, vector spaces that possess gramians 
(vector valued inner products). The main results refer to the case when the kernels 
are invariant under certain actions of $*$-semigroups and show under which conditions 
$*$-representations on VE-spaces, or VH-spaces in the topological case, can be
obtained. Finally we show that these results unify most of dilation type results for
invariant positive semidefinite kernels with operator values as well as recent results on
positive semidefinite maps on $*$-semigroups with values operators from a locally bounded 
topological vector space to its conjugate $Z$-dual space, for $Z$ an ordered $*$-space.
\end{abstract} 

\maketitle 
 
 \section{Introduction}
 
The dilation theory, initiated by the seminal articles of M.A.~Na\u\i mark in \cite{Naimark1} and 
\cite{Naimark2}, consists today of an extraordinary large diversity of results that may look, 
at the first glance, as having next to nothing in common, e.g.\ see N.~Aronszajn \cite{Aronszajn},
W.B.~Arveson \cite{Arveson1}, S.D.~Barreto et al.\ \cite{Barreto}, D.~Ga\c spar and P.~Ga\c spar 
\cite{GasparGaspar1}, \cite{GasparGaspar2}, A.~Gheondea and B.E.~U\u gurcan 
\cite{GheondeaUgurcan}, J.~G\'orniak and A.~Weron \cite{GorniakWeron}, \cite{GorniakWeron2},
J.~Heo \cite{Heo}, G.G.~Kasparov \cite{Kasparov}, R.M.~Loynes \cite{Loynes1}, G.J.~Murphy
\cite{Murphy}, M.~Skeide \cite{Skeide}, W.F.~Stinespring \cite{Stinespring}, F.H.~Szafraniec 
\cite{Szafraniec},  \cite{Szafraniec2}, B.~Sz.-Nagy \cite{SzNagy}, to cite a few only. In a series of 
recent articles  \cite{Gheondea}, \cite{AyGheondea}, and \cite{AyGheondea2}, a unification of this 
theory under some general results of operator valued  positive semidefinite kernels that are 
invariant under actions of $*$-semigroups has been initiated. 
 Historically, based on some classical results on scalar kernels of J.~Mercer \cite{Mercer} and
 A.N.~Kolmogorov \cite{Kolmogorov},
 positive semidefinite kernels that are invariant under actions
 of groups have been used more than forty years ago in mathematical models
 of quantum physics by D.E.~Evans and J.T.~Lewis \cite{EvansLewis} and in probability theory
 by K.R.~Parthasaraty and K.~Schmidt \cite{ParthasaratySchmidt} and turned out to be successful 
even beyond positive semidefiniteness, as in \cite{ConstantinescuGheondea2}. 
 
 On the other hand, positive semidefiniteness of a scalar valued kernel 
 $\fk\colon X\times X\ra \CC$, defined as
 \begin{equation}\label{e:sps}
 \sum_{j,k=1}^n \overline\alpha_i \alpha_j \fk(x_i,x_j)\geq 0,\quad n\in\NN,\ 
 \alpha_1,\ldots,\alpha_n\in \CC,\ x_1,\ldots,x_n\in X,
 \end{equation}
has many different ways of generalisation when comes to operator valued kernels and, 
consequently, the diversity of dilation results increases considerably. From the point of view of
unification of dilation theory there is a challenge: are there a concept of positive 
semidefiniteness  and a concept of vector space where these kernels take values, 
that can yield dilation theorems that are sufficiently general to contain all (most) of the other dilation 
theorems for operator valued kernels or maps? 
Clearly, such a concept of "weakly" positive semidefiniteness must
refer simply to the bare situation as in \eqref{e:sps}, while the concept of vector space 
should be an ordered $*$-spaces and, consequently,
the spaces of dilation that we expect should be,
in the nontopological case, of VE (Vector Euclidean) type or, in the topological case, of 
VH (Vector Hilbert) type, in the sense of R.M.~Loynes \cite{Loynes2}. So far, "weakly" positive
semidefiniteness have been rarely considered, e.g.\ W.L.~Paschke \cite{Paschke} has a remark
on maps on $C^*$-algebras and, for the special purposes of reproducing kernel spaces, 
it was first considered in \cite{GasparGaspar2} and then used in \cite{PaterBinzar}
as well.

The aim of this article is to develop a systematic study of invariant weakly positive semidefinite 
kernels with values in ordered $*$-spaces and to show that most of the previous dilation results
as in \cite{Gheondea}, \cite{AyGheondea}, \cite{AyGheondea2} and hence, most of the known 
dilation theory, 
can be recovered under this setting. The main results are contained in 
theorems \ref{t:veinvkolmo}, \ref{t:vhinvkolmo}, and \ref{t:vhinvkolmow}, from which we then show
how special cases concerning different kinds of "stronger" positive semidefiniteness can be derived.
Of course, since these dilation theorems are so general, in each particular case we expect that
some additional technical difficulties should show up, but here the main 
idea is of unification and there is always a price to be paid in this enterprise.

In the following we briefly present the contents of this article. 
In Section~\ref{s:npr} we briefly recall the general terminology on ordered $*$-spaces and their 
topological versions, on VE-spaces (Vector Euclidean spaces) and VH-spaces 
(Vector Hilbert spaces) and their operator theory. In Lemma~\ref{l:schwarz} we prove a surrogate
of Schwarz inequality which turns out to be very useful. This inequality with the constant $2$ 
has been claimed before in \cite{Ciurdariu}, or simply stated without proof or reference 
as in \cite{PaterBinzar}, but since the proofs we have seen until now turned out to be flawed, 
we give a detailed proof of it with constant $4$. 

Section~\ref{s:wpsk} contains a detailed study of linearisations and reproducing kernel spaces
associated to weakly positive semidefinite kernels which pertains to VE-spaces or VH-spaces.
Since the geometry of these spaces is so badly behaved, a careful treatment is necessary from
the point of view of minimality and of the equivalence of linearisation with reproducing kernel
space. The main results are contained in Section~\ref{s:iwpsk}. In Theorem~\ref{t:veinvkolmo}
we obtain the nontopological fabric of dilation theorems for invariant weakly positive semidefinite
kernels with values in ordered $*$-spaces and then we obtain two topological versions, 
Theorem~\ref{t:vhinvkolmo} for bounded operators and Theorem~\ref{t:vhinvkolmow} 
for continuously adjointable operators. As expected, both of these topological variants 
refer to a variant of the 
B.~Sz.-Nagy boundedness condition but, it is interesting to observe that, a second boundedness 
condition which refers to an anomaly of operator theory for continuously adjointable operators 
on VH-spaces related to the continuity of the adjoint, see condition (c) in 
Theorem~\ref{t:vhinvkolmo2}, does not show up. 

Finally, in Section~\ref{s:usdt} we show that the main theorems contain the dilation results obtained
in \cite{Gheondea}, \cite{AyGheondea}, and \cite{AyGheondea2}, and hence most of the dilation
theory, by explicitly showing how
to put the stage in each case. A special observation is that for the reproducing kernel space 
versions, which is one of the main tool we use, there are some technical difficulties related to
missing a version of Riesz's Representation Theorem in VH-spaces or VE-spaces and which is
solved by carefully using identifications. In addition, we show how the recent results of F.~Pater
and T.~B\^\i nzar \cite{PaterBinzar} on
positive semidefinite maps on $*$-semigroups with values operators from a vector space to its 
conjugate $Z$-dual space, for $Z$ an ordered $*$-space, 
that generalise previous results of J.~G\'orniak and A.~Weron \cite{GorniakWeron2},
can be recovered by our main results, with actually stronger statements.

\section{Notation and Preliminary Results}\label{s:npr}

In this section we review some of the definitions and some theorems on 
ordered $*$-spaces, topologically ordered $*$-spaces, admissible 
spaces, VE-spaces, topologically VE-spaces and VH-spaces, and their operator theory, cf.\
R.M. Loynes, \cite{Loynes1}, \cite{Loynes2} and, 
for a modern treatment of the subject and some proofs, 
we refer also to \cite{AyGheondea} and \cite{AyGheondea2}. 

\subsection{Topologically Ordered $*$-Spaces.}\label{ss:as}
A complex vector space $Z$ is called an \emph{ordered $*$-space} if:
\begin{itemize}
\item[(a1)] $Z$ has an \emph{involution} $*$, that is, a map $Z\ni z\mapsto z^*\in Z$ 
that is \emph{conjugate linear} 
(($s x+t y)^*=\ol s x^*+\ol t y^*$ for all 
$s,t\in\CC$ and all $x,y\in Z$) and \emph{involutive} 
($(z^*)^*=z$ for all $z\in Z$). 
\item[(a2)] In $Z$ there is a \emph{convex cone} $Z_+$ ($s x+t y\in Z_+$ 
for all numbers $s,t\geq 0$ and all $x,y\in Z_+$), that is  
\emph{strict} ($Z_+\cap -Z_+=\{0\}$), and consisting of \emph{selfadjoint elements} 
only  ($z^*=z$ for all 
$z\in Z_+$). This cone is used to define a \emph{partial order} in $Z$ by: 
$z_1  \geq z_2$ if $z_1-z_2 \in Z_+$.
\end{itemize}

The complex vector space $Z$ is called a \emph{topologically ordered $*$-space} if it is 
an ordered $*$-space and:
\begin{itemize}
\item[(a3)] $Z$ is a \emph{Hausdorff separated locally convex space}.
\item[(a4)] The cone $Z_+$ 
is \emph{closed}, with respect to this topology.
\item[(a5)] The topology of $Z$ is \emph{compatible} with the partial ordering in the 
sense that there exists a base of the topology, linearly generated by a family of 
neighbourhoods $\{N_j\}_{j\in\cJ}$ of the origin, such that all of them are 
absolutely convex and 
\emph{solid}, that is, whenever $x\in N_j$ and $0\leq y\leq x$ then $y\in N_j$.
\end{itemize}
It can be proven that axiom (a5) is equivalent with the following one, see \cite{AyGheondea2}:
\begin{itemize}
\item[(a5$^\prime$)] There exists a collection of seminorms 
$\{p_j\}_{j\in \cJ}$ defining the 
topology of $Z$ that are \emph{increasing}, that is, $0\leq x\leq y$ implies 
$p_j(x)\leq p_j(y)$.
\end{itemize}
We denote the collection of all increasing continuous seminorms on $Z$ by $S(Z)$.
$Z$ is called an \emph{admissible space} if, in addition to the axioms (a1)--(a5),
\begin{itemize}
\item[(a6)] The topology on $Z$ is complete.
\end{itemize}

\subsection{Vector Euclidean Spaces and Their Linear Operators.}
Given a complex linear space $\cE$ and an
ordered $*$-space $Z$, a \emph{$Z$-valued inner product} or 
\emph{$Z$-gramian}  is, by definition, a mapping  
$\cE\times \cE\ni (x,y) \mapsto [x,y]\in Z$ subject to 
the following properties:
\begin{itemize}
\item[(ve1)] $[x,x] \geq 0$ for all $x\in \cE$, and $[x,x]=0$ if and only if $x=0$.
\item[(ve2)] $[x,y]=[y,x]^*$ for all $x,y\in\cE$.
\item[(ve3)] $[x,ay_1+by_2]=a[x,y_1]+b[x,y_2]$ for all $a,b\in \mathbb{C}$ and all
  $x_1,x_2\in \cE$.
\end{itemize}

A complex linear space $\cE$ onto which a $Z$-valued inner product 
$[\cdot,\cdot]$ is specified, for a 
certain ordered $*$-space $Z$, is called a \emph{VE-space} 
(Vector Euclidean space) over $Z$.  

In any VE-space $\cE$ over an ordered $*$-space $Z$ the familiar 
\emph{polarisation formula}
\begin{equation}\label{e:polar} 4[x,y]=\sum_{k=0}^3 \iac^k [(x+\iac^k y,x+\iac^k y],
\quad x,y\in \cE,
\end{equation} holds, which shows that the $Z$-valued inner product is perfectly 
defined by the $Z$-valued quadratic map $\cE\ni x\mapsto [x,x]\in Z$.

The concept of \emph{VE-spaces isomorphism} is also naturally defined: this is just 
a linear bijection $U\colon \cE\ra \cF$, for two VE-spaces over the same ordered $*$-space 
$Z$, which is \emph{isometric}, that is, 
$[Ux,Uy]_\cF=[x,y]_\cE$ for all $x,y\in \cE$. 

A useful result for the constructions in this paper is the following lemma.

\begin{lemma}[Loynes \cite{Loynes1}]\label{l:sesqui} 
Let $Z$ be an ordered $*$-space, $\cE$ a complex vector space and
$[\cdot,\cdot]\colon \cE\times \cE \rightarrow Z$ a positive semidefinite 
sesquilinear map, 
that is, $[\cdot,\cdot]$ is linear in the second variable, conjugate linear in the 
first variable, and $[x,x]\geq 0$ for all $x\in \cE$. 
If $f\in \cE$ is such that $[f,f]=0$, then $[f,f']=[f',f]=0$ for all $f' \in \cE$.
\end{lemma}

Given two VE-spaces $\cE$ and $\cF$, over the same ordered $*$-space $Z$, 
one can consider the vector space $\cL(\cE,\cF)$ of all
linear operators $T\colon \cE\ra\cF$. The operator $T$ is 
called \emph{bounded} if there exists $C\geq 0$ such that
\begin{equation}\label{e:bounded} [Te,Te]_\cF\leq C^2 [e,e]_\cE,\quad e\in\cE.
\end{equation} Note that the inequality \eqref{e:bounded} is in the sense of the order 
of $Z$ uniquely determined by the cone $Z_+$, 
see the axiom (a2). 
The infimum of these scalars is denoted by $\|T\|$ and it is called 
the \emph{operator norm} of $T$, more precisely,
\begin{equation}\label{e:opnorm}\|T\|=\inf\{C>0\mid [Te,Te]_\cF\leq C^2 [e,e]_\cE,\mbox{ for all }e\in\cE\}.
\end{equation} Let $\cB(\cE,\cF)$ denote the collection of all 
bounded linear operators $T\colon \cE\ra\cF$. Then $\cB(\cE,\cF)$ is a linear space 
and $\|\cdot\|$ is a norm on it, cf.\ Theorem 1 in \cite{Loynes2}. In addition, if $T$ 
and $S$ are bounded linear operators acting between appropriate VE-spaces over 
the same ordered $*$-space $Z$, then $\|TS\|\leq \|T\| \|S\|$, in particular $TS$ is 
bounded. If $\cE=\cF$ then $\cB(\cE)=\cB(\cE,\cE)$ is a normed algebra, more 
precisely, the operator norm is submultiplicative. 

A linear operator $T\in\cL(\cE,\cF)$ is 
called \emph{adjointable} if there exists $T^*\in\cL(\cF,\cE)$ such that
\begin{equation}\label{e:adj} [Te,f]_\cF=[e,T^*f]_\cE,\quad e\in\cE,\ f\in\cF.
\end{equation} The operator $T^*$, if it exists, is uniquely determined by 
$T$ and called its \emph{adjoint}.
Since an analog of the Riesz Representation Theorem for VE-spaces
does not exist, in general, 
there may be not so many adjointable operators. We denote by 
$\cL^*(\cE,\cF)$ the vector space of all adjointable operators from
$\cL(\cE,\cF)$. 
Note that $\cL^*(\cE)=
\cL^*(\cE,\cE)$ is a $*$-algebra with respect to the involution $*$ determined 
by the operation of taking the adjoint.

An operator $A\in\cL(\cE)$ is called \emph{selfadjoint} if
\begin{equation}\label{e:self} [Ae,f]=[e,Af],\quad e,f\in \cE.
\end{equation} Clearly, any selfadjoint operator $A$ is adjointable and $A=A^*$.
By the polarisation formula \eqref{e:polar}, $A$ is selfadjoint if and only if
\begin{equation}\label{e:self2} [Ae,e]=[e,Ae],\quad e\in\cE.
\end{equation}

%On the other hand, if $A$ is a bounded selfadjoint operator then
%\begin{equation}\label{e:bs} -\|A\| [e,e]\leq [Ae,e]\leq \|A\| [e,e],\quad e\in\cE,
%\end{equation}
%and, conversely, if for some real numbers $m\leq M$ we have
%\begin{equation}\label{e:me}m[e,e]\leq [Ae,e]\leq M[e,e],\quad e\in\cE,
%\end{equation} then $A$ is bounded and selfadjoint. If, in addition, $m$ and $M$ are 
%the minimum, respectively, the maximum values for which \eqref{e:me} holds, then 
%$\|A\|=\max\{|m|,|M|\}$, cf.\ Theorem 3 in \cite{Loynes2}.

An operator $A\in\cL(\cE)$ is \emph{positive} if
\begin{equation}\label{e:pos} [Ae,e]\geq 0,\quad e\in\cE.\end{equation}
Since the cone $Z_+$ consists of selfadjoint elements only, 
any positive operator is selfadjoint and hence adjointable.

Let $\cB^*(\cE)$ denote the collection of all adjointable bounded linear operators 
$T\colon \cE\ra\cE$. Then $\cB^*(\cE)$ is a pre-$C^*$-algebra, that is, it is a 
normed $*$-algebra with the property
\begin{equation}\label{e:prec} \|A^*A\|=\|A\|^2,\quad A\in\cB^*(\cE),
\end{equation} cf.\ Theorem 4 in \cite{Loynes2}.
 In particular, the involution $*$ is isometric on $\cB^*(\cE)$, that is, $\|A^*\|=\|A\|$ 
 for all $A\in\cB^*(\cE)$.
 
 If $A\in\cB^*(\cE)$ can be factored $A=T^*T$, for some $T\in\cB^*(\cE)$, 
 then $A$ is  positive. If, in addition, $\cB^*(\cE)$ is complete, and hence 
 a $C^*$-algebra, and $A\in\cB^*(\cE)$ 
is positive, then $A=T^*T$ for some $T\in\cB^*(\cE)$, 
cf.\ Lemma 2 in \cite{Loynes2}.

\subsection{Vector Hilbert Spaces and Their Linear Operators.}\label{ss:vhs}
If $Z$ is a topologically ordered $*$-space, 
any VE-space $\cE$ can be made in a natural way
into a Hausdorff separated locally convex space by considering the weakest locally 
convex topology on $\cE$ that makes the mapping 
$E\ni h\mapsto [h,h]\in Z$ continuous, more 
precisely, letting $\{N_j\}_{j\in\cJ}$ be the collection of convex and solid 
neighbourhoods of the origin in $Z$ as in axiom (a5), the collection of sets
\begin{equation}\label{e:ujex}U_j=\{x\in \cE\mid [x,x]\in N_j\},\quad j\in\cJ,
\end{equation} is a topological base of neighbourhoods of the origin of $\cE$ that 
linearly generates the weakest locally convex topology on $\cE$ that 
makes all mappings $\cE\ni h\mapsto [h,h] \in Z$ continuous, cf.\ 
Theorem 1 in \cite{Loynes1}. In terms of seminorms, this topology can 
be defined in the following way: let $\{p_j\}_{j\in \cJ}$ be a family of increasing 
seminorms defining the topology of $Z$ as in axiom (a5$^\prime$) and let
\begin{equation}\label{e:qujeh} \tl{p_j}(h)=p_j([h,h])^{1/2},\quad h\in \cE,\ j\in\cJ.
\end{equation}
Then each $\tl{p_j}$ is a seminorm on $\cE$ and its topology 
is fully determined by the family $\{\tl{p_j}\}_{j\in\cJ}$, see Lemma~1.3 in \cite{AyGheondea2}. 
With respect to this topology, we call $\cE$ a 
\emph{topological VE-space} over $Z$. 

We first prove a surrogate of Schwarz Inequality that we will use several times
in this article.

\begin{lemma}\label{l:schwarz}
Let $\cE$ be a topological VE-space over the topologically ordered $*$-space 
$Z$ and $p\in S(Z)$. Then
\begin{equation}\label{e:schwarz}
p([e,f])\leq 4p([e,e])^{1/2}p([f,f])^{1/2}=4\tl{p}(e)\tl{p}(f), \quad e,f\in\cE.
\end{equation}
\end{lemma}

\begin{proof}
For arbitrary $h, k \in \cE$ we have
\begin{equation*}[h \pm k, h \pm k] = [h, h] + [k, k] \pm [h, k] \pm [k, h] \geq 0,
\end{equation*}
in particular,
\begin{equation*} [h, k] + [k, h] \leq [h, h] + [k, k],\end{equation*}
and
\begin{equation}\label{e:zeqo} 
0 \leq [h + k, h + k] \leq [h - k, h - k] + [h + k, h + k] = 2([h, h] + [k, k]).
\end{equation}
Taking into account that $p \in S(Z)$ is increasing, from \eqref{e:zeqo}
it follows that
\begin{equation}\label{e:pehak}
p([h + k, h + k]) \leq 2\bigl(p([h, h]) + p([k, k])\bigr).
\end{equation}
Let now $e, f \in\cE$ be arbitrary. By the polarisation formula \eqref{e:polar} 
and \eqref{e:pehak}, we have
\begin{align*}
p([e, f]) & = p\bigl(\frac{1}{4} \sum_{k=0}^3 \iac^k[e + \iac^kf, e + \iac^kf]\bigr)
 \leq \frac{1}{4} \sum_{k=0}^3 p([e + \iac^kf, e + \iac^kf]) \\
& \leq \frac{2}{4} \sum_{k=0}^3 \bigl(p([e,e]) + p([\iac^kf,  \iac^kf])\bigr)
 = 2\bigl(p([e,e])+p([f,f])\bigr).
\end{align*}
Letting $\lambda > 0$ be arbitrary and changing $e$ with $\sqrt{\lambda}e$ and $f$ with 
$f/\sqrt{\lambda}$ in the previous inequality,
we get
\begin{equation*}
p([e, f]) \leq 2\bigl(\lambda p([e, e]) + \lambda^{-1}p([f, f])\bigr),
\end{equation*}
hence, since the left hand side does not depend on $\lambda$, it follows
\begin{equation*}p([e, f]) \leq \inf_{\lambda>0}2\bigl(\lambda p([e, e]) + \lambda^{-1}p([f, f])\bigr)
 = 4 p([e, e])^{1/2} p([f, f])^{1/2}.\qedhere
\end{equation*}
\end{proof}

Let $\cE$ and $\cF$ be two topological VE-spaces over the same topologically 
ordered $*$-space $Z$. 
Clearly, any bounded linear operator $T\colon\cE\ra\cF$, with definition as in
\eqref{e:opnorm}, is continuous.

If both of $Z$ and $\cE$ are complete with respect to their specified
locally convex topologies, then $\cE$ is called a \emph{VH-space} 
(Vector Hilbert space). Any topological 
VE-space $\cE$ on a topologically ordered $*$-space 
can be embedded as a dense subspace of a VH-space $\cH$, 
uniquely determined up to an isomorphism, cf.\ Theorem 2 in
\cite{Loynes1}. Note that, given two VH-spaces $\cH$ and $\cK$, over the same 
admissible space $Z$, any isomorphism $U\colon \cH\ra\cK$ 
in the sense of VE-spaces, is automatically bounded and adjointable, hence 
$U\in\cB(\cH,\cK)$, and it is natural to call this operator \emph{unitary}.

If $\cE$ and $\cF$ are topological 
VE-spaces over the same admissible space 
$Z$ and $\cF$ is complete, that is, a VH-space, then $\cB(\cE,\cF)$ is a 
Banach space, with respect to the operator norm. In particular, if $\cE$ is a VH-space, 
then $\cB^*(\cE)$ is a $C^*$-algebra. Note that, in this case, the usual notion 
of $C^*$-algebra positive elements in $\cB^*(\cE)$ coincides with that of positive 
operators in the sense of \eqref{e:pos}, \cite{Loynes1}. 

For topological VE-spaces $\cE$ and $\cF$ over the same 
topologically ordered $*$-space $Z$, we denote 
the space of all linear continuous operators $T\colon\cE\ra\cF$ by $\cL_{\mathrm{c}}(\cE,\cF)$, 
and in particular, $\cL_{\mathrm{c}}(\cE,\cE)$ by $\cL_{\mathrm{c}}(\cE)$. The 
$*$-algebra of all continuous and continuously adjointable linear operators 
$T\colon\cE\ra\cF$ are denoted by $\cL^*_{\mathrm{c}}(\cE,\cF)$, and 
$\cL^*_{\mathrm{c}}(\cE)=\cL^*_{\mathrm{c}}(\cE,\cE)$, see \cite{AyGheondea2}.

A subspace $\cM$ of a VH-space $\cH$ is \emph{orthocomplemented} or 
\emph{accessible} if every element $x\in\cH$ can be written as 
$x=y+z$ where $y\in\cM$ and $z$ is such that $[z,m]=0$ for all 
$m\in\cM$, that is, $z$ is in the \emph{orthogonal companion} 
$\cM^{\bot}$ of $\cM$. If such a decomposition exists 
it is unique. Also, any orthocomplemented subspace is closed. A closed subspace $\cM$
of $\cH$ is orthocomplemented if and only if it is the range of a selfadjoint projection, that is,
an adjointable linear operator $P\colon\cH\ra\cH$ such that $P^2=P=P^*$. Note that any
selfadjoint projection is a contraction, in particular it is bounded.

\section{Weakly Positive Semidefinite Kernels}\label{s:wpsk}

\subsection{Hermitian Kernels.}\label{ss:hk}
Let $X$ be a nonempty set and $Z$ an ordered $*$-space. 
A map $\fk\colon X\times X\ra Z$ is called a \emph{$Z$-valued kernel on 
$X$}. If no confusion may arise we also say simply that $\fk$ is a kernel.
The \emph{adjoint} kernel $\fk^*\colon X\times X\ra Z$ is defined by 
$\fk^*(x,y)=\fk(y,x)^*$, for $x,y\in X$. The kernel $\fk$ is 
called \emph{Hermitian} if $\fk^*=\fk$.

Consider $\CC^X$ the complex vector space of all functions $f\colon X\ra \CC$, as 
well as its subspace $\CC^X_0$ consisting of those functions $f\in\CC^X$ with finite 
support.
Given a $Z$-valued kernel $\fk$ on $X$, a pairing $[\cdot,\cdot]_\fk\colon \CC_0^X
\times \CC_0^X\ra Z$ can be defined
\begin{equation}\label{e:gk} [f,g]_\fk=\sum_{x,y\in X} \overline{f(x)}g(y)\fk(x,y),\quad 
f,g\in\CC^X_0.
\end{equation}
The pairing $[\cdot,\cdot]_\fk$ is linear in the second variable and conjugate linear in the 
first variable. If, in addition, $\fk=\fk^*$, then the pairing $[\cdot,\cdot]_\fk$ is 
\emph{Hermitian}, that is,
\begin{equation}\label{e:herm} [f,g]_\fk=[g,f]_\fk^*,\quad f,g\in\CC^X_0.
\end{equation}
Conversely, if the pairing $[\cdot,\cdot]_\fk$ is Hermitian then $\fk=\fk^*$.

A \emph{convolution} operator $K\colon \CC^X_0\ra Z^X$, where $Z^X$ is the 
complex vector space of all functions $g\colon X\ra Z$, can be associated to the 
$Z$-kernel $\fk$ by
\begin{equation}\label{e:conv} (Kg)(x)=\sum_{y\in X}g(y)\fk(x,y),\quad f\in \CC^X_0.
\end{equation}
Clearly, $K$ is a linear operator. A natural relation exists between the paring 
$[\cdot,\cdot]_\fk$ and the convolution operator $K$, more precisely,
\begin{equation}\label{e:pconv} [f,g]_\fk=\sum_{x\in X} \overline{f(x)}(Kg)(x),\quad 
f,g\in\CC^X_0.
\end{equation}
Therefore, it is easy to see from here that the kernel $\fk$ is Hermitian if and only if 
the pairing $[\cdot,\cdot]_\fk$ is Hermitian.

Given a natural number $n$, a $Z$-valued kernel $\fk$ is called \emph{weakly
$n$-positive} 
if for all $x_1,\ldots,x_n\in X$ and all $t_1,\ldots,t_n\in\CC$ we have
\begin{equation}\label{e:np} \sum_{j,k=1}^n \overline{t_k}t_j \fk(x_k,x_j)
\geq 0.
\end{equation}
The kernel $\fk$ is called \emph{weakly positive semidefinite} if it is $n$-positive for all 
$n\in\NN$.

\begin{lemma}\label{l:twopos} Let the $Z$-kernel $\fk$ on $X$ be weakly
$2$-positive. 
Then:

\nr{1} $\fk$ is Hermitian.

\nr{2} If, for some $x\in X$, $\fk(x,x)=0$, then $\fk(x,y)=0$ for all $y\in X$.

\nr{3} There exists a unique decomposition $X=X_0\cup X_1$, 
$X_0\cap X_1=\emptyset$, such that $\fk(x,y)=0$ for all $x,y\in X_0$ and 
$\fk(x,x)\neq 0$ for all $x\in X_1$.
\end{lemma}

\begin{proof} (1) Clearly, weak $2$-positivity implies weak $1$-positivity, 
hence $\fk(x,x)\geq 0$ 
for all $x\in X$. Let $x,y\in X$ be arbitrary. 
Since $\fk$ is weakly $2$-positive, for any $s,t\in\CC$ we have
\begin{equation}\label{e:kxy} |s|^2 \fk(x,x)+|t|^2\fk(y,y)+\overline{s}t 
\fk(x,y)+s\overline{t}\fk(y,x)\geq 0.
\end{equation}
Since the sum of the first two terms in \eqref{e:kxy} is in $Z_+$ and taking into 
account that $Z_+$ consists of selfadjoint elements only, it follows that 
the sum of the last two terms in \eqref{e:kxy} is selfadjoint, that is,
\begin{equation*} \overline{s}t 
\fk(x,y)+s\overline{t}\fk(y,x)=\overline{t}s \fk(x,y)^*
+\overline{s}t \fk(y,x)^*.
\end{equation*} Letting $s=t=1$ and then $s=1$ and $t=\iac$, 
it follows that $\fk(y,x)=\fk(x,y)^*$.

(2) Assume that $\fk(x,x)=0$ and let $y\in X$ be arbitrary. 
From \eqref{e:kxy} it follows that for all $s,t\in\CC$ we have
\begin{equation}\label{e:kxyin} \overline{s}t 
\fk(x,y)+s\overline{t}\fk(y,x) \geq -|t|^2\fk(y,y).
\end{equation}
We claim that for all $s,t\in\CC$ we have
\begin{equation}\label{e:kxyout} \overline{s}t 
\fk(x,y)+s\overline{t}\fk(y,x)=0.
\end{equation}
To prove this, note that for $t=0$ the equality \eqref{e:kxyout} it trivially true. 
If $t\in\CC\setminus\{0\}$, note that we can distinguish two cases: 
first, if $\fk(y,y)=0$, then from \eqref{e:kxyin} it follows $\overline{s}t 
\fk(x,y)+s\overline{t}\fk(y,x)\geq 0$ and then, changing $t$ to $-t$ 
the opposite inequality holds, hence \eqref{e:kxyout}. The second case is $
\fk(y,y)\neq 0$ when we observe that the right hand side in \eqref{e:kxyin} does not depend 
on $s$ hence, replacing $s$ by $ns$, $n\in\ZZ$, a routine reasoning 
shows that \eqref{e:kxyout} must hold as well.

Finally, in \eqref{e:kxyout} we first let $s=1=t$ and then $s=1$ and 
$t=\iac$ and solve for $\fk(x,y)$ which should be $0$.

(3) Denote $X_0=\{x\in X\mid \fk(x,x)=0\}$ and let $X_1=X\setminus X_0$. Then 
use (2) in order to obtain $\fk(x,y)=0$ for all $x,y\in X_0$.
\end{proof}

\subsection{Weak Linearisations}\label{ss:l}
 Given an ordered $*$-space $Z$ and a $Z$-valued kernel $\fk$ on a nonempty set 
 $X$,  a \emph{weak VE-space linearisation}, or \emph{weak Kolmogorov 
decomposition} of $\fk$ is, by definition, a pair $(\cE;V)$, 
subject to the following conditions:
  \begin{itemize}
  \item[(vel1)] $\cE$ is a VE-space over the ordered $*$-space $Z$.
  \item[(vel2)] $V\colon X\ra\cE$ satisfies $\fk(x,y)=[V(x),V(y)]_\cE$ for all $x,y\in X$.
  \end{itemize}
  If, in addition, the following condition holds
  \begin{itemize}
  \item[(vel3)] $\lin V(X)=\cE$,
  \end{itemize}
then the weak VE-space linearisation $(\cE;V)$ is called \emph{minimal}. 

Two weak VE-space linearisations $(V;\cE)$ and $(V';\cE')$ of the same kernel $\fk$ are 
called \emph{unitarily equivalent} if there exists a unitary operator 
$U\colon \cE\ra\cE'$ such that $UV(x)=V'(x)$ for all $x\in X$.

\begin{remarks} \label{r:unique}
(1) Note that any two minimal 
weak VE-space linearisations $(\cE;V)$ and $(\cE^\prime;V^\prime)$ of the same $Z$-kernel 
$\fk$ are unitarily equivalent. The proof
follows in the usual way: if $(\cE';V')$ 
is another minimal weak VE-space linearisation of 
$\fk$, for arbitrary $x_1,\ldots,x_m,y_1,\ldots,y_n\in X$ and arbitrary 
$t_1,\ldots,t_m,s_1,\ldots,s_n\in \CC$, we have
\begin{align*} [\sum_{j=1}^m t_jV(x_j), \sum_{k=1}^m s_k V(y_k)]_\cE 
& = \sum_{j=1}^m \sum_{k=1}^n s_k\ol{t_j}[V(x_j),V(y_k)]_\cE
 = \sum_{k=1}^n\sum_{j=1}^m s_k\ol{t_j}
\fk(x_j,y_k) \\
& =   \sum_{j=1}^m \sum_{k=1}^n s_k\ol{t_j}[V'(x_j),V'(y_k)]_{\cE'} 
 = [\sum_{j=1}^m t_jV'(x_j), \sum_{k=1}^n s_kV'(y_k)]_{\cE'},
\end{align*} hence $U\colon \lin V(X)\ra \lin V'(X)$ defined by 
\begin{equation}\label{e:udef}\sum_{j=1}^m 
t_jV(x_j)\mapsto \sum_{j=1}^m t_jV'(x_j),\quad x_1,\ldots,x_m\in X,\ t_1,\ldots,t_m\in\CC,
\ m\in\NN,\end{equation} 
is a correctly defined linear operator, 
isometric, everywhere defined, and onto. Thus, $U$ is a 
VE-space isomorphism $U\colon \cE\ra\cE'$ and $UV(x)=V'(x)$ for all $x\in X$, by 
construction.

(2) From any weak VE-space linearisation $(\cE;V)$ of $\fk$ one can make a minimal one in a canonical way, more precisely, letting $\cE_0=\lin V(X)$ and 
$V_0\colon X\ra\cE_0$ defined by $V_0(x)=V(x)$, $x\in X$, it follows that 
$(\cE_0;V_0)$ is a minimal weak VE-space linearisation of $\fk$.
\end{remarks}

Let us assume now that $Z$ is an admissible space and $\fk$
is a $Z$-kernel on a set $X$. A \emph{weak VH-space linearisation} of $\fk$ is a
linearisation $(\cH;V)$ of $\fk$ such that $\cH$ is a VH-space. The weak VH-space
linearisation $(\cH;V)$ is called \emph{topologically minimal} if
\begin{itemize}
\item[(vhl3)] $\lin V(X)$ is dense in $\cH$.
\end{itemize}
Two weak VH-space linearisations $(\cH;V)$ and $(\cH';V')$ of the same $Z$-kernel $\fk$ 
are called 
\emph{unitary equivalent} if there exists a unitary operator $U\in\cB^*(\cH,\cH')$ 
such that $UV(x)=V'(x)$ for all $x\in X$. 

\begin{remarks}\label{r:uniqueh}
(a) Any two topologically minimal weak VH-space linearisations of the 
same $Z$-kernel are unitarily equivalent. Indeed, letting $(\cH;V)$ and $(\cH^\prime;V^\prime)$ be 
two minimal weak VH-space linearisations of the $Z$-kernel $\fk$, 
we proceed as in Remark~\ref{r:unique}.(a) and
define $U\colon \lin V(X)\ra\lin V^\prime(X)$ as in \eqref{e:udef}. 
Since $U$ is isometric, it is bounded in the sense of
\eqref{e:bounded}, hence continuous, and then $U$ can be uniquely extended to an isometric
operator $U\colon \cH\ra\cH^\prime$. Since $\lin V^\prime(X)$ is dense in $\cH^\prime$ and $U$
has closed range, it follows that $U$ is surjective, hence $U\in\cB^*(\cH,\cH^\prime)$ is unitary and,
by its definition, see \eqref{e:udef}, we have $UV(x)=V^\prime(x)$ for all $x\in X$. 

(b) From any weak VH-space linearisation $(\cH;V)$ of 
$\fk$ one can make, in a canonical way, a topologically minimal weak VH-space linearisation 
$(\cH_0;V_0)$ by letting $\cH_0=\ol{\lin V(X)}$ and $V_0(x)=V(x)$ for all $x\in X$.
\end{remarks}
  
  \begin{theorem}\label{t:velin} \emph{(a)} Given an ordered $*$-space $Z$ and a 
  $Z$-valued kernel $\fk$ on a nonempty set $X$, 
the following assertions are equivalent:
\begin{itemize} 
\item[(1)] $\fk$ is positive semidefinite.
\item[(2)] $\fk$ admits a weak VE-space linearisation $(\cE;V)$.
\end{itemize}
Moreover, if exists, 
a weak VE-space linearisation $(\cE;V)$ can always be chosen such that 
$\cE\subseteq Z^X$, that is, consisting of functions $f\colon X\ra Z$ only, 
and minimal.

\emph{(b)} If, in addition, $Z$ is an admissible space 
and $\fk\colon X\times X\ra Z$ is a kernel, then any of the assertions 
\emph{(1)} and \emph{(2)} is equivalent with:
\begin{itemize}
\item[(3)] $\fk$ admits a weak VH-space linearisation $(\cH;V)$.
\end{itemize}
Moreover, if exists, 
a weak VH-space linearisation $(\cH;V)$ can always be chosen such that 
$\cH\subseteq Z^X$ and topologically minimal.
\end{theorem}
  
\begin{proof} (1)$\Ra$(2). Assuming that $\fk$ is positive semidefinite, by 
Lemma~\ref{l:twopos}.(1) it follows that $\fk$ is Hermitian, that is, 
$\fk(x,y)^*=\fk(y,x)$ for all $x,y\in X$.
With notation as in Subsection~\ref{ss:hk}, we consider the convolution operator 
$K\colon \CC^X_0\ra Z^X$ and let $Z^X_K$ be its range, more precisely,
\begin{align}\label{e:fezero}  Z^X_K& =\{f\in Z^X\mid f=Kg\mbox{ for some }
g\in\CC^X_0\}\\
& = \{f\in\cF\mid f(x)=\sum_{y\in X} g(y)\fk(x,y)\mbox{ for some }g\in\CC^X_0
\mbox{ and all } y\in X\}.\nonumber
\end{align}
A pairing $[\cdot,\cdot]_{\cE}\colon Z_K^X\times Z_K^X\ra Z$ can be defined by
\begin{equation}\label{e:ipfezero} 
[e,f]_{\cE} =[g,h]_\fk=\sum_{x,y\in X} \overline{g(x)}h(y)\fk(x,y),
\end{equation} where $f=Kh$ and $e=Kg$ for some $g,h\in\CC^X_0$. 
We observe that
\begin{align*} [e,f]_{\cE} & =\sum_{x\in X}\overline{g(x)}f(x)
=\sum_{x,y\in X} \overline{g(x)}\fk(x,y)h(y) = 
\sum_{x,y\in X} h(y)\overline{g(x)}\fk(y,x)^*
=\sum_{x\in X} h(y)e(y)^*,
\end{align*} which shows that the definition in \eqref{e:ipfezero} is correct, that is, 
independent of $g$ and $h$ such that $e=Kg$ and $f=Kh$.

We claim that $[\cdot,\cdot]_{\cE}$ is a $Z$-valued inner product, that is, it 
satisfies all 
the requirements (ve1)--(ve3). The only fact that needs a proof is $[f,f]_{\cE}=0$ 
implies $f=0$. To see this, we use Lemma~\ref{l:sesqui} and first get that 
$[f,f']_{\cE}=0$ for all $f'\in Z^X_K$. For each $x\in X$, let 
$\delta_x \in\CC^X_0$ denote the $\delta$-function with support $\{x\}$,
\begin{equation}\label{e:deltax} 
\delta_x (y)=\begin{cases} 1, & \mbox{ if }y=x, \\ 0, & \mbox{ if }y\neq x. 
\end{cases}
\end{equation}
Letting $f'=K\delta_x $ we have
\begin{equation*} 0=[f,f']_{\cE} =\sum_{y\in X} 
\delta_x f(y)=f(x),
\end{equation*} hence, since $x\in X$ are arbitrary, it follows that 
$f=0$.

Thus, $(Z^X_K;[\cdot,\cdot]_{\cE})$ is a VE-space.
For each $x\in X$ we define $V(x)\in Z^X_K\subseteq\cE$ by
\begin{equation}\label{e:defvex} V(x) =K\delta_x.
\end{equation} Actually, there is an even more explicit way of expressing $V(x)$, 
namely,
\begin{equation}\label{e:vexa} (V(x))(y)=(K\delta_x)(y)
=\sum_{z\in X} \delta_x(z) \fk(y,z)=\fk(y,x),\ \ x\in X.
\end{equation}

On the other hand, for any $x,y\in X$, by \eqref{e:defvex} and \eqref{e:vexa}, we 
have
\begin{equation*} [V(x),V(y)]_\cE=(V(y))(x) =\fk(x,y),
\end{equation*} hence $(\cE;V)$ is a linearisation of $\fk$. We prove that 
it is minimal as well. To see this, note that for any $g\in\CC^X_0$, with notation as in 
\eqref{e:deltax}, we have
\begin{equation*} g=\sum_{x\in \supp(g)} g(x)\delta_x,
\end{equation*} hence, by \eqref{e:defvex}, the linear span of $V(X)$ equals 
$Z^X_K$.

(2)$\Ra$(1). This is proven exactly as in the classical case:
\begin{align*} \sum_{j,k=1}^n t_j\ol{t_k}\fk(x_k,x_j) & = \sum_{j,k=1}^n
t_j\ol{t_k} [V(x_k),V(x_j)]_\cE
 = [\sum_{j=1}^n t_kV(x_k),
\sum_{j=1}^n t_jV(x_j)]_\cE\geq 0,
\end{align*} for all $n\in\NN$, $x_1,\ldots,x_n\in X$, and $t_1,
\ldots,t_n\in \cH$. 

(3)$\Ra$(2). Clear.

(1)$\Ra$(3). Assuming that $Z$ is an admissible space, let 
$\fk$ be positive semidefinite, let $(\cE;V)$ be the 
weak VE-space linearisation of $\fk$. Then, $\cE$ is naturally equipped with a Hausdorff
locally convex topology, see Subsection~\ref{ss:vhs}, and then completed to a VH-
space $\cH$. Thus, $(\cH;V)$ is a weak VH-space linearisation of $\fk$ and it is easy to 
see that it is topologically minimal. The fact that this completion can be made within 
$Z^X$ will follow from Proposition~\ref{p:linrep}.
\end{proof}

\subsection{Reproducing Kernel Spaces}\label{ss:rks}
Let $Z$ be an ordered $*$-space and let $X$ be a nonempty 
set. As in Subsection~\ref{ss:hk}, we consider the complex vector space $Z^X$ 
of all functions $f\colon X\ra Z$. A VE-space $\cR$ over the ordered $*$-space 
$Z$ is called a \emph{weak $Z$-reproducing kernel VE-space on $X$} 
if there exists a Hermitian kernel $\fk\colon X\times X\ra Z$ such that the
following axioms are satisfied:
\begin{itemize} 
\item[(rk1)] $\cR$ is a subspace of $Z^X$, with all algebraic operations.
\item[(rk2)] For all $x\in X$, 
the $Z$-valued map $\fk_x =\fk(\cdot,x)\colon X\ra Z$ belongs to $\cR$.
\item[(rk3)] For all $f\in\cR$ we have $f(x)=[\fk_x,f]_\cR$, for all $x\in X$.
\end{itemize}
The axiom (rk3) is called the \emph{reproducing property} and note that, as a 
consequence, we have
\begin{equation}\label{e:repxy}\fk(x,y)=\fk_y(x)=[\fk_x,\fk_y]_\cR,\quad x,y\in X.
\end{equation}
A weak $Z$-reproducing kernel VE-space $\fk$ on $X$ is called \emph{minimal} if
\begin{itemize}
\item[(rk4)] $\lin\{\fk_x\mid x\in X\}=\cR$.
\end{itemize}

If $Z$ is an admissible space, a weak $Z$-reproducing kernel  
VE-space $\cR$ that is a VH-space 
is called a \emph{weak $Z$-reproducing kernel VH-space}. Such an $\cR$ is called 
\emph{topologically minimal} if
\begin{itemize}
\item[(rk4)$^\prime$] $\lin\{\fk_x\mid x\in X\}$ is dense in $\cR$.
\end{itemize}

\begin{remark}\label{r:minimal}
 Let $\cR$ be a weak $Z$-reproducing kernel VH-space with respect to some 
admissible space $Z$. In general, the closed subspace 
$\ol{\lin\{\fk_x\mid x\in X\}}\subseteq\cR$ may or may not be
orthocomplemented in $\cR$, see Subsection~\ref{ss:vhs}. This anomaly makes
some differences when compared with the classical theory of reproducing kernel spaces,
as is the case in closely related situations as 
in \cite{AyGheondea} and \cite{AyGheondea2} as well.\end{remark}

\begin{proposition}\label{p:mrk}
A weak $Z$-reproducing kernel VH-space $\cR$ with respect to some  admissible space $Z$ is 
topologically minimal if and only if the closed subspace 
$\ol{\lin\{\fk_x\mid x\in X\}}$ is orthocomplemented in $\cR$.
\end{proposition}
\begin{proof}
If $\cM:=\ol{\lin\{\fk_x\mid x\in X\}}$ is orthocomplemented 
then, as a consequence of (rk3), $\cR$ is topologically minimal, in the sense of (rk4)$^\prime$.
Indeed, let $f\in\cR$ be arbitrary. Since $\cM$ is orthocomplemented, 
there exists $f_1\in\cM$ and $f_2\in\cM^{\bot}$ with $f=f_1+f_2$. By 
(rk3) we obtain that $0=[\fk_x,f_2]=f_2(x)$ for all $x\in\cR$, and 
that $f_2=0$. It follows that $f\in\cM$ and $\cM=\cR$, i.e. 
$\lin\{\fk_x\mid x\in X\}$ is dense in $\cR$. The converse implication is trivial.
\end{proof}

We first consider the relation between weak $Z$-reproducing kernel 
VE/VH-spaces and their reproducing kernels.

\begin{proposition}\label{p:vhrepunic} \emph{(a)} Let $\cR$ be a weak $Z$-reproducing 
kernel VE-space on $X$, with respect to some ordered $*$-space $Z$ and 
with kernel $\fk$. Then:
\begin{itemize}
\item[(i)] $\fk$ is positive semidefinite and uniquely determined by $\cR$.
\item[(ii)] $\cR_0=\lin\{\fk_x\mid x\in X\}\subseteq \cR$ is a minimal weak 
$Z$-reproducing kernel VE-space on $X$ and uniquely determined by $\fk$ with this property.
\item[(iii)] The gramian $[\cdot,\cdot]_\cR$ is uniquely determined by $\fk$
on $\cR_0$.
\end{itemize}
\emph{(b)} Assume that $Z$ is admissible and $\cR$ is a weak $Z$-reproducing kernel VH-space.
Then:
\begin{itemize}
\item[(i)] $\ol\cR_0$ is a topologically minimal $Z$-reproducing kernel VH-space in $\cR$.
\item[(ii)] The gramian $[\cdot,\cdot]_\cR$ is uniquely determined by $\fk$
on $\ol\cR_0\subseteq \cR$.
\item[(iii)] If $\cR$ is topologically minimal then it is unique with this property.
\end{itemize}
\end{proposition}

\begin{proof} (a) Let $t_1,\ldots,t_n\in\CC$ and $x_1,\ldots,x_n\in X$ be 
arbitrary. Using \eqref{e:repxy} it follows
\begin{align*} \sum_{j,k=1}^n \ol{t_j}t_k\fk(x_j,x_k)  = 
\sum_{j,k=1}^n \ol{t_j}t_k[\fk_{x_j},\fk_{x_k}]_\cR 
 = [\sum_{j=1}^n t_j\fk_{x_j},\sum_{k=1}^n 
t_k\fk_{x_k}]_\cR\geq 0
\end{align*} hence $\fk$ is positive semidefinite. 
On the other hand, by (rk3) it follows that for all $x\in X$ 
the functions $\fk_x$ 
are  uniquely determined by $(\cR;[\cdot,\cdot]_\cR)$, hence 
$\fk(y,x)=\fk_x(y)$, $x,y\in X$, are uniquely determined. Hence assertion (i) is proven. Assertion (ii)
is clear by inspecting the definitions.
Assertion (iii) is now clear by (rk3), see \eqref{e:repxy}.

(b) The subspace $\ol\cR_0$ of $\cR$ is a topologically minimal $Z$-reproducing kernel VH-space,
by definition. Using the assertion at item (a).(ii) and the continuity of the gramian 
$[\cdot,\cdot]_\cR$, it follows that it is uniquely determined by $\fk$ on $\ol\cR_0$.

Assume that $\cR$ is topologically minimal and let $\cR'$ be another topologically minimal 
weak $Z$-reproducing kernel VH-space on $X$ with the same kernel $\fk$. 
By axiom (rk2) and the property (rk4), $\cR_0=\lin\{\fk_x\mid x\in X\}$ 
is a linear space 
that lies and is dense in both of $\cR$ and $\cR'$. By axiom (rk3), the $Z$-valued 
inner products $[\cdot,\cdot]_\cR$ and $[\cdot,\cdot]_{\cR'}$ coincide on $\cR_0$ 
and then, due to the special way in which the topologies on VH-spaces are defined, see 
\eqref{e:ujex} and \eqref{e:qujeh}, it follows that $\cR$ and $\cR^\prime$ induce the same 
topology on $\cR_0$ hence, taking into account the density of 
$\cR_0$ in both $\cR$ and $\cR^\prime$, we actually have $\cR=\cR'$ as VH-spaces.
\end{proof}

Consequently, given $\cR$ a
weak $Z$-reproducing kernel VE-space on $X$, without any ambiguity
we can talk about \emph{the} $Z$-reproducing kernel $\fk$ corresponding to $\cR$.

As a consequence of Proposition~\ref{p:vhrepunic}, weakly positive semidefiniteness is
an intrinsic property of the reproducing kernel of any weak reproducing kernel VE-space.
In the following we clarify in an explicit fashion 
the relation between weak VE/VH-linearisations and 
weak reproducing kernel VE/VH-spaces associated to positive semidefinite kernels.

\begin{proposition}\label{p:linrep} Let $\fk$ be a weakly
positive semidefinite kernel on $X$ and with values in the ordered $*$-space $Z$.

\emph{(a)} Any weak reproducing kernel VE-space $\cR$ associated to $\fk$
gives rise to a weak VE-space linearisation $(\cE;V)$ of $\fk$, where $\cE=\cR$ and 
\begin{equation}\label{e:vexak}V(x)=\fk_x,\quad x\in X. \end{equation}
If $\cR$ is minimal, then $(\cE;V)$ is minimal.

\emph{(b)} Any minimal weak VE-space linearisation $(\cE;V)$ of $\fk$ gives rise to the minimal
weak reproducing kernel VE-space $\cR$, where
\begin{equation}\label{e:redef} \cR=\{[V(\cdot),h]_\cE\mid h\in\cE\},
\end{equation} that is, $\cR$ consists of all functions $X\ni x\mapsto 
[V(x),e]_\cK\in Z$, for all $e\in\cE$, in 
particular, $\cR\subseteq Z^X$ and $\cR$ is endowed with the algebraic 
operations inherited from the complex vector space $Z^X$.
\end{proposition}

\begin{proof} (a)
Assume that $(\cR;[\cdot,\cdot]_\cR)$ is a weak $Z$-reproducing kernel 
VE-space on $X$, with reproducing kernel $\fk$. We let $\cE=\cR$ and define $V$ 
as in \eqref{e:vexak}.
Note that $V(x)\in\cE$ for all $x\in X$. 
Also, by \eqref{e:repxy} we have
\begin{equation*}[V(x),V(y)]_\cE=\fk(x,y),\quad x,y\in X.
\end{equation*} Thus, $(\cE;V)$ is a weak VE-space 
linearisation of $\fk$.

(b) Let $(\cE;V)$ be a minimal weak VE-space linearisation of $\fk$. 
Let $\cR$ be defined by \eqref{e:redef},
that is, $\cR$ consists of all functions $X\ni x\mapsto 
[V(x),h]_\cE\in Z$, in particular $\cR\subseteq Z^X$ with all algebraic 
operations inherited from the complex vector space $Z^X$. 

The correspondence
\begin{equation}\label{e:defuv} \cE\ni h\mapsto Uh=[V(\cdot),h]_\cE\in\cR
\end{equation} is clearly surjective.
In order to verify that it is injective as well, let $h,g\in\cE$ be such that 
$[V(\cdot),h]_\cE=[V(\cdot),g]_\cE$. Then, for all $x\in X$ we have
\begin{equation*}[V(x),h]_\cE=[V(x),g]_\cE,
\end{equation*} equivalently,
\begin{equation}\label{e:vexeh}[V(x),h-g]_\cE=0,\quad x\in X.
\end{equation} By the minimality of the linearisation $(\cE;V)$ it follows 
that $g=h$. Thus, $U$ is a bijection.

Clearly, the bijective map $U$ defined at \eqref{e:defuv} is linear, hence a linear 
isomorphism of complex vector spaces $\cE\ra \cR$. On $\cR$ we introduce a 
$Z$-valued pairing
\begin{equation}\label{e:ufege}
[Uf,Ug]_\cR=[f,g]_\cE,\quad f,g\in\cE.
\end{equation} Since $(\cE;[\cdot,\cdot]_\cE)$ is a VE-space over $Z$, it follows that 
$(\cR;[\cdot,\cdot]_\cR)$ is a VE-space over $Z$. Indeed, this follows from the 
observation that, by \eqref{e:ufege}, we 
transported the $Z$-gramian from $\cE$ to $\cR$ or, in other words, 
we have defined 
on $\cR$ the $Z$-gramian that makes the linear isomorphism $U$ a unitary operator 
between the VE-spaces $\cE$ and $\cR$.

We show that $(\cR;[\cdot,\cdot]_\cR)$ is a weak $Z$-reproducing kernel VE-space with 
corresponding reproducing kernel $\fk$.
By definition, $\cR\subseteq Z^X$. On the other hand, since 
\begin{equation*}\fk_x(y)=\fk(y,x)=[V(y),V(x)]_\cE,\mbox{ for all }x,y\in X,
\end{equation*} taking into account that $V(x)\in\cE$, by \eqref{e:redef} it 
follows that $\fk_x\in\cR$ for all $x\in X$. Further, for all $f\in\cR$ and all $x\in X$ 
we have
\begin{equation*}[\fk_x,f]_\cR =[\fk_x,[V(\cdot),g]_\cE ]_\cR=[V(x),g]_\cE,
\end{equation*} where $g\in\cE$ is the unique vector such that $[V(\cdot),g]_\cE=f$, 
which shows 
that $\cR$ satisfies the reproducing axiom as well. Finally, taking into account the 
minimality of the linearisation $(\cE;V)$ and the definition 
\eqref{e:redef}, 
it follows that $\lin\{ \fk_x\mid x\in X\}=\cR$. Thus, 
$(\cR;[\cdot,\cdot]_\cR)$ is a minimal weak $Z$-reproducing kernel VE-space with 
reproducing kernel $\fk$. 
\end{proof}

\begin{proposition}\label{p:vhrep}
Let $\fk$ be a weakly positive semidefinite kernel on $X$ and valued in the
admissible space $Z$.

\emph{(a)} Any weak reproducing kernel VH-space $\cR$ associated to $\fk$
gives rise to a weak VH-space linearisation $(\cH;V)$ of $\fk$, where $\cH=\cR$ and 
\begin{equation}\label{e:vexakvh}V(x)=\fk_x,\quad x\in X.
\end{equation} If $\cR$ is topologically minimal then $(\cH;V)$ is topologically minimal.

\emph{(b)} Any topologically minimal 
weak VH-space linearisation $(\cH;V)$ of $\fk$ gives rise to the topologically minimal
weak reproducing kernel VH-space $\cR$, where
\begin{equation}\label{e:redefvh} \cR=\{[V(\cdot),h]_\cH\mid h\in\cH\},
\end{equation} that is, $\cR$ consists of all functions $X\ni x\mapsto 
[V(x),e]_\cK\in Z$, for all $e\in\cH$, in 
particular, $\cR\subseteq Z^X$ and $\cR$ is endowed with the algebraic 
operations inherited from the complex vector space $Z^X$. 
\end{proposition}

\begin{proof} (a) The argument is similar to that used to prove assertion (a)
of Proposition~\ref{p:linrep}.

(b) Let $(\cH;V)$ be a topologically minimal weak VH-space linearisation of 
$\fk$ and let 
$\cR$ be defined as in \eqref{e:redefvh}. 
The correspondence
\begin{equation}\label{e:defuvh} \cH\ni h\mapsto Uh=[V(\cdot),h]_\cH\in\cR
\end{equation} is a linear bijection $U\colon \cH\ra\cR$. 
The argument to support this claim is similar with that used during the
proof of item (b) in Proposition~\ref{p:linrep}, with the difference that from \eqref{e:vexeh} we
the topological minimality of the linearisation $(\cH;V)$ in order to conclude
that $g=h$. Thus, $U$ is a bijection.

On $\cR$ we introduce a 
$Z$-valued pairing as in \eqref{e:ufege}
 Since $(\cH;[\cdot,\cdot]_\cH)$ is a VH-space over $Z$, it follows that 
$(\cR;[\cdot,\cdot]_\cR)$ is a VH-space over $Z$. This follows from the 
observation that, by \eqref{e:ufege}, we 
transported the $Z$-gramian from $\cH$ to $\cR$ or, in other words, 
we have defined 
on $\cR$ the $Z$-gramian that makes the linear isomorphism $U$ a unitary operator 
between the VH-spaces $\cH$ and $\cR$.

Finally, $(\cR;[\cdot,\cdot]_\cR)$ is the topologically minimal 
weak $Z$-reproducing kernel VH-space with 
corresponding reproducing kernel $\fk$. The argument is again similar with that used in
the proof of item (b) in Proposition~\ref{p:linrep}, with the difference that here we use
the topological minimality.
\end{proof}

The following theorem adds one more characterisation of positive semidefinite 
kernels, when compared to Theorem~\ref{t:velin}, in terms of reproducing kernel 
spaces. It's proof is a direct 
consequence of Proposition~\ref{p:linrep}, Proposition~\ref{p:vhrep}, and Theorem~\ref{t:velin}.

\begin{theorem}\label{t:vhrepker} \emph{(a)} Let $Z$ be an ordered $*$-space, 
$X$ a nonempty set, and $\fk\colon X\times X\ra Z$ a Hermitian kernel. 
The following assertions are equivalent:
\begin{itemize}
\item[(1)] $\fk$ is weakly positive semidefinite.
\item[(2)] $\fk$ is the $Z$-valued reproducing kernel of a VE-space $\cR$ in $Z^X$.
\end{itemize}

\emph{(b)} If, in addition, $Z$ is an admissible space then assertions (1) and (2) are 
equivalent with
\begin{itemize}
\item[(3)] $\fk$ is the $Z$-valued reproducing kernel of a VH-space $\cR$ in $Z^X$.
\end{itemize}
In particular, any weakly positive semidefinite $Z$-valued kernel $\fk$ 
has a topologically minimal weak 
$Z$-reproducing kernel VH-space $\cR$, uniquely determined by $\fk$.
\end{theorem}

As a consequence of the last assertion of Theorem~\ref{t:vhrepker}, 
given $\fk\colon X\times X\ra Z$ a positive semidefinite kernel for an 
admissible space $Z$, we can denote, without any ambiguity, 
by $\cR_\fk$ the unique topologically minimal  weak $Z$-reproducing kernel VH-space on $X$ associated to 
$\fk$.

\section{Invariant Weakly Positive Semidefinite Kernels}\label{s:iwpsk}
Let $X$ be a nonempty set, a (multiplicative) semigroup $\Gamma$, and 
an action of $\Gamma$ on $X$, denoted by $\xi\cdot x$, for all $\xi\in\Gamma$ 
and all $x\in X$. By definition, we have 
\begin{equation}\label{e:action}
\alpha\cdot(\beta\cdot x)=(\alpha\beta)\cdot x\mbox{ for all }\alpha,
\beta\in \Gamma\mbox{ and all }x\in X.\end{equation} 
This means that we have a semigroup morphism 
$\Gamma\ni\xi\mapsto \xi\cdot \in G(X)$, where $G(X)$ denotes the semigroup, with 
respect to composition, of all maps $X\ra X$. 
In case the semigroup $\Gamma$ has a unit 
$\epsilon$, the action is called \emph{unital} if $\epsilon\cdot x=x$ for all $x\in X$, 
equivalently, $\epsilon\cdot=\mathrm{Id}_X$.

We assume further that $\Gamma$ is a $*$-semigroup, that is, there is an 
\emph{involution} $*$ on $\Gamma$; this means that $(\xi\eta)^*=\eta^* \xi^*$ and 
$(\xi^*)^*=\xi$ for all $\xi,\eta\in\Gamma$. Note that, in case $\Gamma$ has a unit 
$\epsilon$ then $\epsilon^*=\epsilon$.

\subsection{Invariant Weak VE-Space Linearisations.}\label{ss:ivesl}
Given an ordered $*$-space $Z$ we are interested in those Hermitian kernels 
$\fk\colon X\times X\ra Z$ that are \emph{invariant} under the action of 
$\Gamma$ on $X$, that is,
\begin{equation}\label{e:invariant} 
\fk(y,\xi\cdot x)=\fk(\xi^*\cdot y,x)\mbox{ for all }x,y\in X
\mbox{ and all }\xi\in\Gamma.
\end{equation}
A triple $(\cE;\pi;V)$ is called a \emph{invariant weak VE-space linearisation} 
of the $Z$-valued kernel $\fk$ and the action of $\Gamma$ on $X$, if:
\begin{itemize}
\item[(ivel1)] $(\cE;V)$ is a weak VE-space linearisation of the kernel $\fk$.
\item[(ivel2)] $\pi\colon \Gamma\ra\cL^*(\cE)$ is a $*$-representation, that is, a 
multiplicative $*$-morphism.
\item[(ivel3)] $V$ and $\pi$ are related by the formula: 
$V(\xi\cdot x)= \pi(\xi)V(x)$, for all $x\in X$, $\xi\in\Gamma$.
\end{itemize}
Let $(\cE;\pi;V)$ be an invariant weak VE-space linearisation of the kernel 
$\fk$. Since $(\cE;V)$ is a weak linearisation
and taking into account the axiom (ivel3), for all $x,y\in X$ and all
$\xi\in\Gamma$, we have
\begin{align}\label{e:keyxi}
\fk(y,\xi\cdot x)  & =[V(y),V(\xi\cdot x)]_\cE=[V(y),\pi(\xi)V(x)]_\cE \\
&= [\pi(\xi^*)V(y),V(x)]_\cE= [V(\xi^*\cdot y),V(x)]_\cE=\fk(\xi^*\cdot y,x),\nonumber
\end{align} hence $\fk$ is invariant under the action of $\Gamma$ on $X$.

If, in addition to the axioms (ivel1)--(ivel3), the triple $(\cE;\pi;V)$ has also the 
property
\begin{itemize}
\item[(ivel4)] $\lin V(X)=\cE$,
\end{itemize} that is, the linearisation $(\cE;V)$ is minimal,
then $(\cE;\pi;V)$ is called \emph{minimal}.

\begin{remarks}\label{r:min} 
(a) The minimality condition (ivel4) does not depend on the representation $\pi$. 
Apparently, $\lin\pi(\Gamma)V(X)$ looks like a suitable 
candidate to replace $\lin V(X)$ in (ivel4). However,
in case the $*$-semigroup has a unit, then we have $\lin\pi(\Gamma)V(X)=\lin V(X)$ but, when
$\Gamma$ does not have a unit, only the inclusion $\lin\pi(\Gamma)V(X)\subseteq \lin V(X)$ holds 
and hence, $\lin\pi(\Gamma)V(X)$ may be too small to accommodate all $V(x)$ for $x\in X$.

(b) Let $(\cE;\pi;V)$ be an invariant weak VE-space linearisation of the kernel 
$\fk$. Then, for each $\gamma\in\Gamma$, we $\pi(\gamma)V(x)=V(\gamma\cdot x)$, for all
$x\in X$, hence $\pi(\gamma)$ leaves invariant $\lin V(X)$ and, consequently, letting 
$\cE_0=\lin V(X)$, $\pi_0\colon \Gamma\ra \cL^*(\cE_0)$ defined by 
$\pi_0(\gamma)f=\pi(\gamma) f$ 
for all $f\in \lin V(X)$, and $V_0\colon X\ra \cE_0$ defined by $V_0(x)=V(x)$, it follows that
$(\cE_0;\pi_0;V_0)$ is a minimal invariant weak VE-space linearisation of the kernel $\fk$.
\end{remarks}

As usually \cite{SzNagy}, minimal invariant VE-space linearisations preserve linearity.

\begin{proposition}\label{p:vhinvkolmolinear} Assume that, given 
an ordered $*$-space $Z$ valued kernel $\fk$, 
invariant under the action of the $*$-semigroup 
$\Gamma$ on $X$, for some fixed $\alpha,\beta,\gamma\in\Gamma$ we have 
$\fk(y,\alpha\cdot x)+\fk(y,\beta\cdot x)=\fk(y,\gamma\cdot x)$ for all $x,y\in X$. Then 
for any minimal weak invariant VE-space linearisation $(\cE;\pi;V)$ of $\fk$, the 
representation satisfies $\pi(\alpha)+\pi(\beta)=\pi(\gamma)$.
\end{proposition}

\begin{proof} For any $x,y\in X$ we have
\begin{align*} [(\pi(\alpha)+\pi(\beta))V(x),V(y)]_\cE & = [\pi(\alpha)V(x)+\pi(\beta) 
V(x),V(y)]_\cE \\
& = \fk(\alpha\cdot x,y)+\fk(\beta\cdot x,y) \\ & =\fk(\gamma\cdot x,y)
= [\pi(\gamma)V(x),V(y)]_\cE 
\end{align*} hence, since $V(X)$ linearly spans $\cE$, it follows that 
$\pi(\alpha)+\pi(\beta)=\pi(\gamma)$.
\end{proof}

Two invariant weak VE-space linearisations $(\cE;\pi;V)$ and $(\cE';\pi^\prime;V^\prime)$, of
the same Hermitian kernel $\fk$, are called \emph{unitarily equivalent} if there exists
a unitary operator $U\colon \cE\ra\cE'$ such that $U\pi(\xi)=\pi'(\xi)U$ for all $\xi\in\Gamma$,  
and $UV(x)=V'(x)$ for all $x\in X$. Let us note that, in case both of these invariant
weak VE-space linearisations are minimal, then this is equivalent with the
requirement that the weak VE-space linearisations $(\cE;V)$ and $(\cE';V')$ are unitary
equivalent. 

Here have the first main theorem of this article in which invariant weakly positive semidefinite 
kernels are characterised by invariant weak VE-space linearisations and by certain 
$*$-representations on weak $Z$-reproducing kernel VE-spaces. 

\begin{theorem}\label{t:veinvkolmo} Let $\Gamma$ be a $*$-semigroup that 
acts on the nonempty set $X$, and let $\fk\colon X\times X\ra Z$ be a 
$Z$-valued kernel for some ordered $*$-space $Z$. 
The following assertions are equivalent:
\begin{itemize}
\item[(1)] $\fk$ satisfies the following conditions:
\begin{itemize}
\item[(a)] $\fk$ is weakly positive semidefinite.
\item[(b)] $\fk$ is invariant under the action of $\Gamma$ on $X$, that is, 
\eqref{e:invariant} holds.
\end{itemize}
\item[(2)] $\fk$ has an invariant weak VE-space linearisation $(\cE;\pi;V)$.
\item[(3)] $\fk$ admits a weak $Z$-reproducing kernel VE-space $\cR$ and 
there exists a $*$-representation $\rho\colon \Gamma\ra\cL^*(\cR)$ such that 
$\rho(\xi)\fk_x=\fk_{\xi\cdot x}$ for all $\xi\in\Gamma$, $x\in X$.
\end{itemize}

Moreover, in case any of the assertions \emph{(1)}, \emph{(2)}, or \emph{(3)} holds, 
then a minimal invariant weak VE-space linearisation of $\fk$ can be constructed 
and a minimal weak $Z$-reproducing kernel $\cR$ as in \emph{(3)} can constructed as well.
\end{theorem}

\begin{proof} (1)$\Ra$(2). We consider the notation and the minimal weak VE-space linearisation 
$(\cE;V)$ constructed as in the proof of the implication (1)$\Ra$(2) 
of Theorem~\ref{t:velin}. For each $\xi\in\Gamma$ we let 
$\pi(\xi)\colon Z^X\ra Z^X$ be 
defined by
\begin{equation}\label{e:pixi} 
(\pi(\xi)f)(y)=f(\xi^*\cdot y),\quad f\in Z^X,\ y\in X,\ \xi\in\Gamma.
\end{equation}
We claim that $\pi(\xi)$ leaves $Z^X_K$ invariant, where $K$ is the convolution 
operator defined at \eqref{e:conv} and $Z^X_K\subseteq Z^X$ denotes its range. 
To see this, let $f\in Z^X_K$, that is, $f=Kg$ for 
some $g\in \CC^X_0$ or, even more explicitly, by \eqref{e:fezero},
\begin{equation}f(y)=\sum_{x\in X} g(x)\fk(x,y),\quad y\in X.
\end{equation}
Then, 
\begin{equation}\label{e:fexistar} 
f(\xi^*\cdot y)  =\sum_{x\in X}g(x)\fk(x,\xi^*\cdot y)
 =\sum_{x\in X}g(x)\fk(\xi
\cdot x,y)=\sum_{z\in X} g_\xi(z)\fk(z,y),
\end{equation} where,
\begin{equation*} g_\xi(z)=\begin{cases} 0,& \mbox{ if }\xi\cdot x=z
\mbox{ has no solution }x\in X,\\ \sum\limits_{\xi\cdot x=z} g(x), & \mbox{ otherwise.}
\end{cases}
\end{equation*} Since clearly $g_\xi \in \CC^X_0$, that is, $g_\xi$ has finite support, 
it follows 
that $\pi(\xi)$ leaves $Z^X_K$ invariant. In the following we denote by the same 
symbol $\pi(\xi)$ the map $\pi(\xi)\colon Z^X_K\ra Z^X_K$.

In the following we prove that $\pi$ is a representation of the semigroup $\Gamma$ 
on the complex vector space $Z^X_K$, that is, 
\begin{equation}\label{e:piab}
\pi(\alpha\beta)f=\pi(\alpha)\pi(\beta)f,\quad \alpha,\beta\in \Gamma,\ f\in Z^X_K.
\end{equation}
To see this, let $f\in Z^X_K$ be fixed and denote $h=\pi(\beta)f$, that is, 
$h(y)=f(\beta^*\cdot y)$ for all $y\in X$. Then $\pi(\alpha)\pi(\beta)f=\pi(\alpha)h$, 
that is, $(\pi(\alpha)h)(y)=h(\alpha^*\cdot y)=f(\beta^*\alpha^*\cdot y)
=f((\alpha\beta)^*\cdot y)=(\pi(\alpha\beta))(y)$, for all $y\in X$, which proves 
\eqref{e:piab}.

Next we show that $\pi$ is actually a $*$-representation, that is,
\begin{equation}\label{e:piastar} [\pi(\xi)f,f']_{\cE}=[f,\pi(\xi^*)f']_{\cE},\quad 
f,f'\in Z^X_K.
\end{equation}
To see this, let $f=Kg$ and $f'=Kg'$ for some $g,g'\in\CC^X_0$. Then, by
\eqref{e:ipfezero} and \eqref{e:fexistar}, 
\begin{align*} [\pi(\xi)f,f']_{\cE} & = \sum_{y\in X}g'(y)f(\xi^*\cdot y) 
= \sum_{x,y\in X}g'(y) \ol{g(x)}\fk(\xi^*\cdot y,x) \\
& = \sum_{x,y\in X}g'(y)\ol{g(x)}\fk(y,\xi\cdot x) 
 = \sum_{x\in X}\ol{g(x)}f'(\xi\cdot x)^*
 = [f,\pi(\xi^*)f']_\cE,
\end{align*} and hence the formula \eqref{e:piastar} is proven.

In order to show that the axiom (vel3) holds as well, we use \eqref{e:vexa} 
and \eqref{e:pixi}. 
Thus, for all $\xi\in\Gamma$, $x,y\in X$ and taking into account that $\fk$ is 
invariant under the action of $\Gamma$ on $X$, we have
\begin{align}\label{e:vexic} 
(V(\xi\cdot x))(y) & =\fk(\xi\cdot x,y)=\fk(x,\xi^*\cdot y) \\
 & =(V(x))(\xi^*\cdot y)=(\pi(\xi) V(x))(y),\nonumber
\end{align} which proves (vel3). Thus, $(\cE;\pi;V)$, here constructed, 
is an invariant weak VE-space linearisation of the Hermitian kernel $\fk$.
Note that $(\cE;\pi;V)$ is minimal, that is, the axiom (vel4) holds, since the 
linearisation $(\cE;V)$ is minimal, by the proof of Theorem~\ref{t:velin}.

In order to prove the uniqueness of the minimal weak invariant linearisation, 
let $(\cK^\prime;\pi^\prime;V^\prime)$ be another minimal invariant weak VE-space 
linearisation of $\fk$. We 
consider the unitary operator $U\colon \cK\ra\cK'$ defined as in \eqref{e:udef} and 
we already know that $UV(x)=V'(x)$ for all $x\in X$. Since, for any $\xi\in\Gamma$, 
$x\in X$ we have
\begin{equation*}U\pi(\xi)V(x)=UV(\xi\cdot x)=V'(\xi\cdot x)=\pi'(\xi)V'(x)
=\pi'(\xi)UV(x),
\end{equation*} and taking into account the minimality, it follows that 
$U\pi(\xi)=\pi'(\xi)U$ for all $\xi\in\Gamma$.

(2)$\Ra$(1). Let $(\cE;\pi;V)$ be an invariant weak VE-space linearisation of the kernel 
$\fk$. We already know from the proof of Theorem~\ref{t:velin} that $\fk$ 
is positive semidefinite and it was shown in \eqref{e:keyxi} that $\fk$ is invariant 
under the action of 
$\Gamma$ on $X$.

(2)$\Ra$(3). Let $(\cE;\pi;V)$ be an invariant weak VE-space linearisation of the kernel 
$\fk$ and the action of $\Gamma$ on $X$. Without loss of generality, we can 
assume that it is minimal. Indeed, since we have already proven the implication (2)$\Ra$(1),
we observe that during the proof of the implication (1)$\Ra$(2), we obtained a minimal invariant
weak VE-space linearisation of $\fk$.

We use the notation and the facts established during the 
proof of the implication (2)$\Ra$(3) in Theorem~\ref{t:vhrepker}. Then, 
for all $x,y\in X$ we have
\begin{equation*} \fk_{\xi\cdot x}(y)=\fk(y,\xi\cdot x)=[V(y),V(\xi\cdot x)]_\cK
=[V(y),\pi(\xi) V(x)]_\cK,
\end{equation*} hence, letting $\rho(\xi)=U\pi(\xi)U^{-1}$, where 
$U\colon\cK\ra\cR$ is 
the unitary operator defined as in \eqref{e:defuv}, we obtain a $*$-representation of 
$\Gamma$ on the VE-space $\cR$ such that $\fk_{\xi\cdot x}=\rho(\xi)\fk_x$ for all 
$\xi\in\Gamma$ and $x\in X$.

(3)$\Ra$(2). Let $(\cR;\rho)$, where $\cR$ is a weak $Z$-reproducing kernel 
VE-space of $\fk$ and $\rho\colon\Gamma\ra\cL^*(\cR)$ is a $*$-representation 
such that 
$\fk_{\xi\cdot x}=\rho(\xi)\fk_x$ for all $\xi\in\Gamma$ and $x\in X$. As in the proof of 
the 
implication (3)$\Ra$(2) in Theorem~\ref{t:vhrepker}, we show that $(\cR;V)$, where 
$V$ is defined as in \eqref{e:vexak}, is a minimal linearisation of $\fk$. 
Letting $\pi=\rho$, it is then easy to see that $(\cR;\pi;V)$ is an invariant weak VE-space 
linearisation of the kernel $\fk$ and the action of $\Gamma$ on $X$.
\end{proof}

\subsection{Boundedly Adjointable Invariant Weak VH-Space Linearisations.}\label{ss:biwvhsl}
Let us assume now that $Z$ is an admissible space and 
$\fk\colon X\times X\rightarrow Z$ is a kernel.
A triple $(\cK;\pi;V)$ is called a \emph{boundedly adjointable  invariant weak VH-space linearisation} 
of the $Z$-valued kernel $\fk$ and the action of $\Gamma$ on $X$, if:
\begin{itemize}
\item[(ivhl1)] $(\cK;V)$ is a weak VH-space linearisation of the kernel $\fk$.
\item[(ivhl2)] $\pi\colon \Gamma\ra\cB^*(\cK)$ is a $*$-representation, that 
is, a multiplicative $*$-morphism.
\item[(ivhl3)] $V$ and $\pi$ are related by the formula: 
$V(\xi\cdot x)= \pi(\xi)V(x)$, for all $x\in X$, $\xi\in\Gamma$.
\end{itemize}
Let $(\cK;\pi;V)$ be a boundedly adjointable  invariant weak VH-space linearisation of the kernel 
$\fk$. As in \eqref{e:keyxi}, it follows that $\fk$ is invariant under the action of 
$\Gamma$ on $X$.

If, in addition to the axioms (ivhl1), (ivhl2), and
(ivhl3), the triple $(\cK;\pi;V)$ has also the 
property
\begin{itemize}
\item[(ivhl4)] $\lin V(X)$ is dense in $\cK$,
\end{itemize} that is, the weak VH-space linearisation $(\cH;V)$ is topologically minimal,
then $(\cK;\pi;V)$ is called \emph{topologically minimal}. 
A similar observation as in Remark~\ref{r:min} can be made: in case $\Gamma$ has a unit then
(ivhl4) is equivalent with saying $\lin\pi(\Gamma)V(X)$ is dense in $\cK$ but, in general the 
apparently more candidate $\lin\pi(\Gamma)V(X)$ is too small to provide a suitable topological 
minimality condition.

Two boundedly adjointable  invariant weak 
VH-space linearisations $(\cK;\pi;V)$ and $(\cK^\prime;\pi^\prime;V^\prime)$ 
of the same kernel $\fk$ are \emph{unitarily invariant} if there exists a unitary 
$U\in\cB^*(\cK,\cK^\prime)$ such that $U\pi(\xi)=\pi'(\xi)U$ for all
$\xi\in\Gamma$ and $UV(x)=V'(x)$ for all $x\in X$. Let us note that, in case both of these 
boundedly adjointable  invariant weak VH-space linearisations are topologically minimal then they are 
unitarily equivalent.

The analog of Proposition~\ref{p:vhinvkolmolinear} for topologically
minimal invariant weak VH-space linearisations holds as well.

\begin{proposition}\label{p:vhinvkolmolinearmt} Assume that, given 
an admissible space $Z$  and a $Z$-valued kernel $\fk$, 
invariant under the action of the $*$-semigroup 
$\Gamma$ on $X$, for some fixed $\alpha,\beta,\gamma\in\Gamma$ we have 
$\fk(y,\alpha\cdot x)+\fk(y,\beta\cdot x)=\fk(y,\gamma\cdot x)$ for all $x,y\in X$. Then, 
for any topologically minimal boundedly adjointable  
invariant weak VH-space linearisation $(\cK;\pi;V)$ of $\fk$, the 
representation satisfies $\pi(\alpha)+\pi(\beta)=\pi(\gamma)$.
\end{proposition}

\begin{proof} The same argument as in the proof of Proposition~\ref{p:vhinvkolmolinear} applies
with the small difference that we use the topological minimality and get the 
same conclusion.
\end{proof}

Here we have the second main theorem of this article in which 
invariant weakly positive semidefinite kernels
are characterised by boundedly adjointable invariant weak VE-space linearisations and by 
certain $*$-representations with boundedly adjointable operators on 
weak $Z$-reproducing kernel VE-spaces. This is the
first topological analogue of Theorem~\ref{t:veinvkolmo}.

\begin{theorem}\label{t:vhinvkolmo} Let $\Gamma$ be a $*$-semigroup that 
acts on the nonempty set $X$, and let $\fk\colon X\times X\ra Z$ be a $Z$-valued 
kernel for some admissible space $Z$. 
The following assertions are equivalent:
\begin{itemize}
\item[(1)] $\fk$ satisfies the following conditions:
\begin{itemize}
\item[(a)] $\fk$ is weakly positive semidefinite.
\item[(b)] $\fk$ is invariant under the action of $\Gamma$ on $X$, that is, 
\eqref{e:invariant} holds.
\item[(c)] For any $\alpha\in \Gamma$ there exists $c(\alpha)\geq 0$ such that
\begin{equation}\label{e:cbound} \sum_{j,k=1}^n t_j\ol{t_k}\fk(\alpha\cdot x_k,\alpha\cdot 
x_j) \leq c(\alpha)^2 \sum_{j,k=1}^n t_j\ol{t_k}\fk(x_k,x_j),
\end{equation} for $n\in\NN$, 
all $x_1,\ldots,x_n\in X$, and all $t_1,\ldots,t_n\in \CC$.
\end{itemize}
\item[(2)] $\fk$ has a boundedly adjointable  invariant weak VH-space linearisation $(\cK;\pi;V)$.
\item[(3)] $\fk$ admits a weak $Z$-reproducing kernel VH-space $\cR$ and 
there exists a $*$-representation $\rho\colon \Gamma\ra\cB^*(\cR)$ such that 
$\rho(\xi)\fk_x=\fk_{\xi\cdot x}$ for all $\xi\in\Gamma$, $x\in X$.
\end{itemize}

Moreover, in case any of the assertions \emph{(1)}, \emph{(2)}, or \emph{(3)} holds, 
then a topologically minimal boundedly adjointable  invariant weak VH-space linearisation can be constructed
and a topologically minimal weak $Z$-reproducing kernel VH-space $\cR$ as in assertion \emph{(3)} can be constructed as well.
\end{theorem}

\begin{proof} (1)$\Ra$(2). We consider the notation and the minimal invariant weak 
VE-space linearisation $(\cE;\pi;V)$ constructed as in the proof of the implication 
(1)$\Ra$(2) of Theorem~\ref{t:veinvkolmo}. 
Considering $Z^X_K$ as a VE-space with $Z$-gramian 
$[\cdot,\cdot]_\cE$, we consider its natural topology as in Subsection~\ref{ss:vhs} 
and we prove now that $\pi(\xi)$ is bounded for all $\xi
\in \Gamma$. Indeed, let $f=Kg$ for some $g\in \CC^X_0$. 
Using the definition of $\pi(\xi)$ 
and the boundedness condition (c), we have
\begin{align*} [\pi(\xi)f,\pi(\xi)f]_{\cK} & = [\pi(\xi^*)\pi(\xi)f,f]_{\cK} = [\pi(\xi^*
\xi)f,f]_{\cK} \\
& = \sum_{x,y\in X}g(y)\ol{g(x)}\fk(\xi^*\xi\cdot y,x)
 = \sum_{x,y\in X}g(y)\ol{g(x)}\fk(\xi\cdot y,\xi\cdot x) \\
& \leq c(\xi)^2 \sum_{x,y\in X}g(y)\ol{g(x)}\fk(y,x)  = c(\xi)^2 [f,f]_{\cK},
\end{align*} and hence the boundedness of $\pi(\xi)$ is proven. 
This implies that $\pi(\xi)$ can be uniquely 
extended by continuity to an operator $\pi(\xi)\in\cB(\cK)$. In addition,
since $\pi(\xi^*)$ also extends by continuity to an operator $\pi(\xi^*)\in\cB(\cK)$ and 
taking into account \eqref{e:piastar}, it follows that $\pi(\xi)$ is adjointable and 
$\pi(\xi^*)=\pi(\xi)^*$. We conclude that $\pi$ is a $*$-representation of $\Gamma$ 
in $\cB^*(\cK)$, that is, the axiom (ivhl2) holds.

The uniqueness of the topologically minimal boundedly adjointable  invariant weak VH-space 
linearisation follows as usually.

(2)$\Ra$(1). Let $(\cK;\pi;V)$ be a boundedly adjointable  invariant weak  VH-space linearisation of the kernel 
$\fk$. We already know from the proof of Theorem~\ref{t:velin} that $\fk$ 
is positive 
semidefinite and it was shown in \eqref{e:keyxi} that $\fk$ is invariant under 
the action of 
$\Gamma$ on $X$. In order to show that the boundedness condition (c) holds 
as well, 
let $\alpha\in \Gamma$, $n\in\NN$, 
$x_1,\ldots,x_n\in X$, and $t_1,\ldots,t_n\in\CC$ be arbitrary. 
Then
\begin{align*} \sum_{j,k=1}^n \ol{t_k}t_j\fk(\alpha\cdot x_k,\alpha\cdot 
x_j)& = \sum_{j,k=1}^n \ol{t_k}t_j[\pi(\alpha^*)\pi(\alpha)V(x_k),V(x_j)]_\cK 
\\ & = \sum_{j,k=1}^n\ol{t_k}t_j [\pi(\alpha)V(x_k),\pi(\alpha)V(x_j)]_\cK \\
& = [\pi(\alpha)\sum_{k=1}^n t_k V(x_k),\pi(\alpha)\sum_{j=1}^n t_j V(x_j)]_\cK \\
& \leq \|\pi(\alpha)\|^2 [\sum_{k=1}^n t_k V(x_k),\sum_{j=1}^n t_j V(x_j)]_\cK \\
& = \|\pi(\alpha)\|^2 \sum_{j,k=1}^n \ol{t_k}t_j\fk(x_k,x_j),
\end{align*} and hence (c) holds with $c(\alpha)=\|\pi(\alpha)\|\geq 0$.

(2)$\Ra$(3). Let $(\cK;\pi;V)$ be a boundedly adjointable  invariant weak VH-space linearisation of the kernel 
$\fk$ with respect to the action of $\Gamma$ on $X$. Without loss of generality we can 
assume that it is topologically minimal. Indeed, since we have already proven the implication 
(2)$\Ra$(1),
we observe that during the proof of the implication (1)$\Ra$(2), we obtained a topologically
minimal invariant weak VH-space linearisation of $\fk$.
 
We use the notation and the facts established during the 
proof of the implication (2)$\Ra$(3) in Theorem~\ref{t:vhrepker}. Then, 
for all $x,y\in X$ we have
\begin{equation*} \fk_{\xi\cdot x}(y)=\fk(y,\xi\cdot x)=[V(y),V(\xi\cdot x)]_\cK
=[V(y),\pi(\xi) V(x)]_\cK,
\end{equation*} hence, letting $\rho(\xi)=U\pi(\xi)U^{-1}$, where 
$U\colon\cK\ra\cR$ is 
the unitary operator defined as in \eqref{e:defuv}, we obtain a $*$-representation of 
$\Gamma$ on the VH-space $\cR$ such that $\fk_{\xi\cdot x}=\rho(\xi)\fk_x$ for all 
$\xi\in\Gamma$ and $x\in X$.

(3)$\Ra$(2). Let $(\cR;\rho)$, where $\cR=\cR(\fk)$ is the weak reproducing kernel 
VH-space of $\fk$ and $\rho\colon\Gamma\ra\cB^*(\cR)$ is a $*$-representation 
such that 
$\fk_{\xi\cdot x}=\rho(\xi)\fk_x$ for all $\xi\in\Gamma$ and $x\in X$. As in the proof of 
the 
implication (3)$\Ra$(2) in Theorem~\ref{t:vhrepker}, we show that $(\cR;V)$, where 
$V$ is defined as in \eqref{e:vexak}, is a minimal weak linearisation of $\fk$. 
Letting $\pi=\rho$, it is then easy to see that $(\cR;\pi;V)$ is a boundedly adjointable  invariant weak VH-space 
linearisation of the kernel $\fk$ with respect to the action of $\Gamma$ on $X$.
\end{proof}

\subsection{Continuously Adjointable Invariant Weak VH-Space Linearisations.}
\label{ss:caiwvhsl}
Let $Z$ be an admissible space. A triple $(\cK;\pi;V)$ is called a \emph{continuously adjointable 
invariant weak VH-space linearisation} 
of the $Z$-valued kernel $\fk$ and the action of $\Gamma$ on $X$, if the requirements (ivhl1) and
(ivhl2) holds and, instead of (ihvl2), it satisfies
\begin{itemize}
\item[(ivhl2)$^\prime$] $\pi\colon \Gamma\ra\cL^*_{\mathrm{c}}(\cK)$ is a $*$-representation, 
that is, a multiplicative $*$-morphism.
\end{itemize}

Clearly,  for any continuously adjointable invariant weak VH-space linearisation $(\cK;\pi;V)$
of the kernel $\fk$, it follows that $\fk$ is invariant under the action of 
$\Gamma$ on $X$.

If, in addition to the axioms (ivhl1), (ivhl2)$^\prime$, and
(ivhl3), the triple $(\cK;\pi;V)$ has also the 
property (ivhl4), that is, the weak VH-space linearisation $(\cH;V)$ is topologically minimal,
then $(\cK;\pi;V)$ is called a \emph{topologically minimal continuously adjointable
invariant weak VH-space linearisation} of  the $Z$-kernel $\fk$ with respect to the action of 
$\Gamma$ on $X$.

The unitary equivalence of two continuously adjointable invariant weak 
VH-space linearisations $(\cK;\pi;V)$ and $(\cK^\prime;\pi^\prime;V^\prime)$ 
of the same kernel $\fk$ is defined as in the case of boundedly adjointable  invariant weak VH-space 
linearisations and their topological minimality implies their unitary equivalence.

The analog of Proposition~\ref{p:vhinvkolmolinear} for topologically
minimal continuously adjointable invariant weak VH-space linearisations holds as well.

The next theorem is the analogue of Theorem \ref{t:vhinvkolmo} for continuously adjointable
invariant weak VH-space linearisations in which
the boundedness condition $1.(c)$ of Theorem \ref{t:vhinvkolmo} is replaced with
a weaker one.
%such weak boundedness conditions appear also thm of ...
\begin{theorem}\label{t:vhinvkolmow} %weak version of the main theorem
Let $\Gamma$ be a $*$-semigroup that 
acts on the nonempty set $X$, and let $\fk\colon X\times X\ra Z$ be a $Z$-valued 
kernel for some admissible space $Z$. 
The following assertions are equivalent:
\begin{itemize}
\item[(1)] $\fk$ satisfies the following conditions:
\begin{itemize}
\item[(a)] $\fk$ is weakly positive semidefinite.
\item[(b)] $\fk$ is invariant under the action of\, $\Gamma$ on $X$, that is, 
\eqref{e:invariant} holds.
\item[(c)] For any $\alpha\in \Gamma$ and any seminorm $p\in S(Z)$, 
there exists a seminorm $q\in S(Z)$ and a constant $c(\alpha)\geq 0$ such that
\begin{equation}\label{e:cboundw} p(\sum_{j,k=1}^n t_j\ol{t_k}\fk(\alpha\cdot x_k,\alpha\cdot 
x_j)) \leq c(\alpha)^2 q(\sum_{j,k=1}^n t_j\ol{t_k}\fk(x_k,x_j)),
\end{equation} for $n\in\NN$, 
all $x_1,\ldots,x_n\in X$, and all $t_1,\ldots,t_n\in \CC$.
\end{itemize}
\item[(2)] $\fk$ has a continuously adjointable invariant weak VH-space linearisation $(\cK;\pi;V)$.
\item[(3)] $\fk$ admits a weak $Z$-reproducing kernel VH-space $\cR$ and 
there exists a $*$-representation $\rho\colon \Gamma\ra\cL^*_{\mathrm{c}}(\cR)$ such that 
$\rho(\xi)\fk_x=\fk_{\xi\cdot x}$ for all $\xi\in\Gamma$, $x\in X$.
\end{itemize}

Moreover, in case any of the assertions \emph{(1)}, \emph{(2)}, or \emph{(3)} holds, 
then a topologically minimal continuously adjointable invariant VH-space linearisation 
can be constructed 
and a topologically minimal weak $Z$-reproducing kernel VH-space $\cR$ as in assertion 
\emph{(3)} can be constructed as well.
\end{theorem}

\begin{proof}
(1)$\Ra$(2). We consider the notation and constructions as in the proof of the implication 
(1)$\Ra$(2) of Theorem~\ref{t:veinvkolmo}, and follow the same 
idea as in the proof of the implication (1)$\Ra$(2) of Theorem~\ref{t:vhinvkolmo}, 
with the difference that the weak boundedness condition 1.(c) is used. 
For any $\xi\in\Gamma$, $f=Kg$ and $p\in S(Z)$ 
there exist $q\in S(Z)$ and $c(\xi)\geq 0$ such that
\begin{align*} p([\pi(\xi)f,\pi(\xi)f]_{\cK}) & = p(\sum_{x,y\in X}g(y)\ol{g(x)}\fk(\xi\cdot y,\xi\cdot x)) \\
& \leq c(\xi)^2 q(\sum_{x,y\in X}g(y)\ol{g(x)}\fk(y,x)) 
 = c(\xi)^2 q([f,f]_{\cK}),
\end{align*} hence the continuity of $\pi(\xi)$ is proven. 
This implies that $\pi(\xi)$ can be uniquely 
extended by continuity to an operator $\pi(\xi)\in\cL_{\mathrm{c}}(\cK)$. In addition,
since $\pi(\xi^*)$ also extends by continuity to an operator $\pi(\xi^*)\in\cL_{\mathrm{c}}(\cK)$ and 
taking into account \eqref{e:piastar}, it follows that $\pi(\xi)$ is adjointable and 
$\pi(\xi^*)=\pi(\xi)^*$. We conclude that $\pi$ is a $*$-representation of $\Gamma$ 
in $\cL^*_{\mathrm{c}}(\cK)$.

The uniqueness of the topologically minimal continuously adjointable invariant VH-space 
linearisation follows as usually.

(2)$\Ra$(1). By the proof of the implication (2)$\Ra$(1) of Theorem \ref{t:vhinvkolmo}, we only 
have to show that the boundedness condition (c) holds. 
Let $\alpha\in \Gamma$, $n\in\NN$, 
$x_1,\ldots,x_n\in X$, and $t_1,\ldots,t_n\in\CC$ be arbitrary. 
Then, due to the continuity of $\pi(\alpha)$ and taking into account the $S(Z)$ is directed,
there exist $q\in S(Z)$ and $c(\alpha)\geq 0$ such that
\begin{align*} p(\sum_{j,k=1}^n \ol{t_k}t_j\fk(\alpha\cdot x_k,\alpha\cdot 
x_j))& = p([\pi(\alpha)\sum_{k=1}^n t_k V(x_k),\pi(\alpha)\sum_{j=1}^n t_j V(x_j)]_\cK) \\ 
& \leq c(\alpha)^2 q([\sum_{k=1}^n t_k V(x_k),\sum_{j=1}^n t_j V(x_j)]_\cK) \\
& = c(\alpha)^2 q(\sum_{j,k=1}^n \ol{t_k}t_j\fk(x_k,x_j)).
\end{align*} 

(2)$\Ra$(3). Let $(\cK;V;\pi)$ be a continuously adjointable weak VH-space linearisation
of the kernel $\fk$ with respect to the action of $\Gamma$ on $X$.
Using exactly the same ideas in the proof of the implication (2)$\Ra$(1) of Theorem \ref{t:vhinvkolmo},
we obtain a continuous $*$-representation of 
$\Gamma$ on the VH-space $\cR$ defined by $\rho(\xi)=U\pi(\xi)U^{-1}$, where 
$U\colon\cK\ra\cR$ is 
the unitary operator defined as in \eqref{e:defuv}.

(3)$\Ra$(2). Let $(\cR;\rho)$, where $\cR=\cR(\fk)$ is the weak reproducing kernel 
VH-space of $\fk$ and $\rho\colon\Gamma\ra\cL_{\mathrm{c}}^*(\cR)$ is a $*$-representation 
such that 
$\fk_{\xi\cdot x}=\rho(\xi)\fk_x$ for all $\xi\in\Gamma$ and $x\in X$. As in the proof of 
the 
implication (3)$\Ra$(2) in Theorem~\ref{t:vhinvkolmo}, 
letting $\pi=\rho$, it is then easy to see that $(\cR;\pi;V)$ is a weak VH-space 
linearisation of the kernel $\fk$ and $\pi$ satisfies the required properties.
\end{proof}

\section{Unification of Some Dilation Theorems}\label{s:usdt}%unification

In this section we show how various dilation theorems can be obtained as special cases of
Theorem~\ref{t:veinvkolmo}, Theorem~\ref{t:vhinvkolmo}, and Theorem~\ref{t:vhinvkolmow}.

%choose a kernel and admissible space / ordered *-space Z

\subsection{Invariant Kernels with Values Adjointable Operators.}
We show that Theorem 2.8 in \cite{AyGheondea} can be seen as a special case of Theorem 
\ref{t:veinvkolmo}. We first recall necessary definitions 
from \cite{AyGheondea}.

In this subsection we will consider a kernel on a nonempty set $X$
and taking values in
$\cL^*(\cH)$, for a VE-space $\cH$
over an ordered $*$-space $Z$, that is, 
a map $\fl\colon X \times X \ra \cL^*(\cH)$. 

A kernel $\fl\colon X \times X \ra \cL^*(\cH)$ is called
\emph{positive semidefinite} if  
for all $n \in \NN$, 
$x_1,x_2,\cdots,x_n \in X$,
and  $h_1,h_2,\cdots,h_n \in \cH$, we have
\begin{equation}\label{e:hpsd}%H pos semidef
\sum_{i,j=1}^{n}[\fl(x_i,x_j)h_j,h_i]_{\cH}\geq 0.
\end{equation}

An \emph{invariant $\cL^*(\cH)$-valued VE-space linearisation} of a kernel $\fl$ 
and an action of a $*$-semigroup $\Gamma$
on $X$ is, by definition, a triple $(\tl{\cE};\tl{\pi};\tl{V})$ such that 
\begin{itemize}
\item[(hvel1)]$\tl{\cE}$ is a VE-space
over the same ordered $*$-space $Z$,
\item[(hvel2)]$\tl{\pi}\colon\Gamma \ra \cL^*(\tilde\cE)$
is a $*$-representation, 
\item[(hvel3)]$\tl{V}\colon X \ra \cL^*(\cH,
\tl{\cE})$, satisfying $\fk(x,y)=\tl{V}(x)^*\tl{V}(y)$ for all $x,y \in X$ and 
$\tl{V}(\xi\cdot x)=\tl{\pi}(\xi)\tl{V}(x)$
for all $x \in X$, $\xi \in \Gamma$.
\end{itemize}
If an invariant $\cL^*(\cH)$-valued VE-space
linearisation has the property that
$\lin V(X)\cH=\tl{\cE}$,
then it is called \emph{minimal}. Two invariant $\cL^*(\cH)$-VE-space linearisations 
$(\tl{\cE};\tl{\pi};\tl{V})$ and $(\tl{\cF};\tl{\rho};\tl{W})$ of the same kernel $\fl$ are called 
\emph{unitarily equivalent} if there exists a unitary operator $U:\tl{\cE} \ra \tl{\cF}$
such that $U\tl{\pi}(\gamma)=\tl{\rho}(\gamma)U$ for all $\gamma\in\Gamma$ and
$U\tl{V}(x)=\tl{W}(x)$ for all $x \in X$. 

Let $\cH^X$ be the vector space of all maps $f\colon X\ra\cH$, 
for a nonempty set $X$ and a VE-space $\cH$ over the 
ordered $*$-space $Z$. A VE-space $\tl{\cR}$ over the same ordered $*$-space $Z$ is
called a \emph{$\cL^*(\cH)$-reproducing kernel VE-space} on $X$ of the kernel $\fl$ if 
\begin{itemize}
\item[(hrk1)] $\tl{\cR}$
is a vector subspace of $\cH^X$.
\item[(hrk2)] For all $x \in X$ and $h \in \cH$, the $\cH$-valued function 
$\fl_{x}h:=\fl(\cdot, x)h$ belongs to $ \tl{\cR}$.
\item[(hrk3)] For all $f \in \tl{\cR}$
we have $[f(x),h]_{\cH}=[f,\fl_{x}h]_{\tl{\cR}}$ 
for all $x \in X$ and $h \in \cH$.
\end{itemize}
The space $\tl{\cR}$ is \emph{minimal} if $\tl{\cR}=\lin\{\fl_xh\mid x\in X,\ h\in\cH\}$.

\begin{theorem}[Theorem 2.8 in \cite{AyGheondea}] \label{t:vevergh}
Let $\Gamma$ be a $*$-semigroup acting on a nonempty set $X$, 
$\cH$ be a VE-space on an ordered $*$-space $Z$, 
and $\fl\colon X \times X \ra \cL^*(\cH)$ be
a kernel. The following assertions are equivalent:
\begin{itemize}
\item[(1)] $\fl$ satisfies the following properties:
\begin{itemize}
\item[(a)] $\fl$ is positive semidefinite.
\item[(b)] $\fl$ is invariant under the action of $\Gamma$ on $X$. 
\end{itemize}
\item[(2)] $\fl$ has an invariant $\cL^*(\cH)$-valued VE-space linearisation 
$(\tl{\cE};\tl{\pi};\tl{V})$.
\item[(3)] $\fl$ admits a $\cL^*(\cH)$-reproducing kernel
VE-space $\tl{\cR}$ and there exists a
$*$-representation $\tl{\rho}: \Gamma \ra \cL^*(\tl{\cR})$ such that
$\tl{\rho}(\xi) \fl_x h = \fl_{\xi\cdot x}h$ 
for all $\xi \in \Gamma$, $x \in X$, $h \in \cH$.
\end{itemize}
Moreover, in case any of the assertions (1), (2) or (3)
holds, a minimal invariant $\cL^*(\cH)$-VE-space linearisation 
can be constructed, and a pair 
$(\tl{\cR};\tl{\rho})$ as in (3) with $\tl{\cR}$ can be always obtained as well.
\end{theorem}

\begin{proof}
(1)$\Ra$(2). Define a kernel $\fk\colon (X\times\cH)\times(X\times\cH)\ra Z$ by 
\begin{equation*}
\fk((x,h),(y,g)):=[\fl(y,x)h,g]_{\cH},\quad x,y\in X,\ h,g\in\cH.
\end{equation*}
 Since $\fl$ is semipositive definite in the sense of \eqref{e:hpsd}, 
$\fk$ is weakly positive semidefinite:
\begin{align*}
\sum_{k,j=1}^{n}\ol{t_k}t_j\fk((x_k,h_k),(x_j,h_j))
&=\sum_{k,j=1}^{n}\ol{t_k}t_j[\fl(x_j,x_k)h_k,h_j] 
=\sum_{k,j=1}^{n}[\fl(x_j,x_k)t_kh_k,t_jh_j]\geq 0
\end{align*}
for all $n\in\NN$, $\{x_j\}_{j=1}^{n}\in X$, $\{h_j\}_{j=1}^{n}\in\cH$ and 
$\{t_j\}_{j=1}^{n}\in\CC$.

Define an action of $\Gamma$ on $(X\times\cH)$ in the following way: 
$\xi\cdot(x,h)=(\xi\cdot x,h)$ for all $\xi\in\Gamma$, $x\in X$ and $h\in\cH$. 
Using the $\Gamma$ invariance of $\fl$ it follows that
$\fk$ is $\Gamma$ invariant: letting $\xi\in\Gamma$, $x,y\in X$ and $g,h\in\cH$ 
we have 
\begin{align*}
\fk(\xi\cdot(x,h),(y,g))&=[\fl(y,\xi\cdot x)h,g]=[\fl(\xi^*\cdot y, x)h,g] 
=\fk((x,h),\xi^*(y,g)).
\end{align*}

By Theorem \ref{t:veinvkolmo}, there exists a minimal weak VE-space linearisation 
$(\cE;\pi;V)$ of $\fk$ and the action of $\Gamma$ on $(X\times\cH)$. By construction, 
see \eqref{e:vexa},
it is clear that $V(x,h)$ depends linearly on $h\in\cH$, therefore, for each $x\in X$ 
a linear operator of VE-spaces $\tilde V(x)\colon\cH\ra\cE$ can be defined by $\tilde V(x)h=V(x,h)$. 

We now have 
$[\tilde V(x)h,\tilde V(y)g]_{\cE}=\fk((x,h),(y,g))=[\fl(y,x)h,g]_{\cH}$ for all $x,y\in X$ 
and $h,g\in\cH$. By the minimality of $\cE$, it follows that 
$\tilde V(x)$ is an adjointable operator with $\tilde V(y)^*\tilde V(x)=\fl(y,x)$ 
for all $x,y\in X$. 

On the other hand, we have $\pi(\xi)V(x,h)=V(\xi\cdot x,h)=\tilde V(\xi\cdot x)h$ 
for all $h\in\cH$ 
and hence $\pi(\xi)\tilde V(x)=\tilde V(\xi\cdot x)$ for all $\xi\in\Gamma$ and $x\in X$, 
showing that $(\cE;\pi;\tilde V)$ is a minimal invariant $\cL^*(\cH)$-valued VE-space linearisation 
of the kernel $\fl$ and the action of $\Gamma$ on $X$.

(2)$\Ra$(3). Let $(\tl{\cE};\tl{\pi};\tl{V})$ be an invariant $\cL^*(\cH)$-valued VE-space linearisation 
of the kernel $\fl$, hence $\fl(x,y)=\tl{V}(x)^*\tl{V}(y)$ for all $x,y\in X$.
Define $V\colon(X\times\cH)\ra\tl{\cE}$ by 
\begin{equation}\label{e:vexah}
V(x,h)=\tl{V}(x)h,\quad x\in X,\ h\in\cH.
\end{equation} We also 
have 
\begin{equation}\label{e:tep}
\tl{\pi}(\xi)V(x,h)=\tl{\pi}(\xi)\tl{V}(x)h=\tl{V}(\xi\cdot x)h=V(\xi\cdot x,h),\quad
\xi\in\Gamma,\ x\in X,\ h\in\cH,
\end{equation} hence $\tl{\pi}(\xi)$ leaves $\tl{\cE}_0=\lin V(X,\cH)$ invariant
for all $\xi\in\Gamma$. In the following, we denote by the same symbol
$\tl{\pi}\colon\Gamma\ra\cL^*(\tl{\cE}_0)$, the $*$-representation viewed as
$\tl{\pi}(\gamma)\colon \tl{\cE}_0\ra\tl{\cE}_0$ for all $\gamma\in\Gamma$. Then 
$(\tl{\cE}_0;\tl{\pi};V)$ is a minimal invariant weak VE-space 
linearisation for the kernel $\fk\colon(X\times\cH)\times(X\times\cH)\ra Z$ 
defined by 
\begin{align*}
\fk((x,h),(y,g)) & =[V(x,h),V(y,g)]_{\tl{\cE}} \\
& = [\tl{V}(x)h,\tl{V}(y)g]_{\tl{\cE}} = [h,\tl{V}(x)^*\tl{V}(y)g]_{\tl{\cE}} \nonumber \\
& = [h,\fl(x,y)g]_\cH,\quad x,y\in X,\ h,g\in\cH,\nonumber
\end{align*} 
and the action of $\Gamma$ on $(X\times\cH)$ given by 
\begin{equation}\label{e:xic}
\xi\cdot(x,h)=(\xi\cdot x,h),\quad \xi\in\Gamma,\ x\in X,\ h\in\cH.
\end{equation}

By Theorem~\ref{t:veinvkolmo}, there exists a minimal weak $Z$-reproducing kernel VE-space 
$\cR\subseteq Z^{X\times \cH}$, with reproducing kernel $\fk$, 
and a $*$-representation $\rho\colon\Gamma\ra\cL^*(\cR)$ 
such that $\rho(\xi)\fk_{(x,h)}=\fk_{\xi\cdot(x,h)}$ 
for all $\xi\in\Gamma$, $x\in X$, $h\in\cH$. As the proof of Theorem~\ref{t:veinvkolmo} shows,
without loss of generality we can assume that $\cR$ is the collection of all maps $X\times \cH\ra Z$ 
defined by $X\times\cH\ni (x,h)\mapsto [\tl{V}(x)h,f]_{\tl{\cE}}$, 
where $f\in\tl{\cE}_0$,
which provides an identification of $\cR$ with $\tl{\cE}_0$ by the formula
\begin{equation}\label{e:fexaha}
f(x,h)=[V(x,h),f]_\cR=[\tl{V}(x)h,f]_{\tl{\cE}}=[h,\tl{V}(x)^*f]_\cH,\quad h\in\cH.
\end{equation}
Consequently, for each $f\in\cR$ and $x\in X$,
there exists a unique vector $\tilde f(x)=\tl{V}(x)^*f\in\cH$ such that
\begin{equation}\label{e:fexah}
f(x,h)=[h,\tilde f(x)]_\cH,\quad h\in\cH,
\end{equation}
which gives rise to a map $\cR\ni f\mapsto\tl{f}\in \cH^X$. Let 
$\tl{\cR}$ be the vector space of all $\tl{f}$, for $f\in\cR$. Since, by the reproducing property
of the kernel $\fk$ and \eqref{e:fexah} we have 
\begin{equation*}
[\fk_{(x,h)},\fk_{(y,g)}]_\cR=\fk_{(y,g)}(x,h)=[h,\widetilde{\fk_{(y,g)}}(x)]_\cH,\quad h,g\in\cH,\ x,y\in X,
\end{equation*}
taking into account the reproducing property of the kernel $\fl$, 
it follows that $\fl_{x}h=\widetilde{\fk_{(x,h)}}\in\tl{\cR}$ for all $x\in X$ and $h\in\cH$. 

It is easy to see that the map 
$U\colon\cR\ni f \ra \tl{f}\in\tl{\cR}$ is linear, one-to-one, and onto. 
Therefore, defining $[\tl{f},\tl{g}]_{\tl{\cR}}:= [f,g]_{\cR}$ 
makes $\tl{\cR}$ a VE-space, and $U$ becomes a unitary 
operator of VE-spaces. Defining $\tl{\rho}:= U\rho U^*$, the pair
$(\tl{\cR},\tl{\rho})$ has all the required properties.

(3)$\Ra$(1). Assume that $(\tl{\cR};\tl{\rho})$ is a pair consisting of an
$\cL^*(\cH)$-reproducing kernel VE-space of $\fl$ and a $*$-representation 
$\tl{\rho}:\Gamma\ra\cL^*(\tl{\cR})$ such that 
$\rho(\xi)\fl_{x}h=\fl_{\xi\cdot x}h$ for all $\xi\in\Gamma$, 
$x\in X$, $h\in\cH$. We have 
\begin{align*}
\sum_{i,j=1}^n [\fl(x_i,x_j)h_j,h_i]&=\sum_{i,j=1}[\fl_{x_j}h_j(x_i),h_i] 
=\sum_{i,j=1}^n[\fl_{x_j}h_j,\fl_{x_i}h_i] 
=[\sum_{j=1}\fl_{x_j}h_j,\sum_{i=1}\fl_{x_i}h_i]\geq 0
\end{align*}
for all $n\in\NN$, $\{ x_i \}_{i=1}^n\in X$, $\{ h_i \}_{i=1}^n\in\cH$. 
Therefore $\fl$ is positive semidefinite in the sense of \eqref{e:hpsd}. Moreover, by (hrk3)
\begin{align*}
[\fl(x,\xi\cdot y)h,g] =[\fl_{\xi\cdot y}h(x),g]=[\tl{\rho}(\xi)\fl_{y}h(x),g]
 =[\tl{\rho}(\xi)\fl_{y}h,\fl_{x}g]=[\fl_{y}h,\tl{\rho}(\xi^*)\fl_{x}g]=[\fl(\xi^*x,y)h,g],
\end{align*}
for all $x,y\in X$ and $g,h\in\cH$, and 
the invariance of the kernel $\fl$ is proven. 
\end{proof}

\begin{remark}\label{r:tif}
The crucial point in the proof of the implication (2)$\Ra$(3) of Theorem~\ref{t:vevergh} is the
proof of \eqref{e:fexah} which we obtained as a consequence of the identification of $\cR$ with 
$\tilde\cE_0$. In the following we show that there is a direct proof of \eqref{e:fexah}, without using
this identification.

By minimality, $\cR=\lin\{\fk_{(x,h)}\mid x\in X,\ h\in\cH\}$ so let 
$f=\sum_{j=1}^n\alpha_j\fk_{(y_j,g_j)}$ for some $n\in\NN$, $y_1,\ldots,y_n\in X$, 
$g_1,\ldots,g_n\in\cH$ and $\alpha_1,\ldots,\alpha_n\in\CC$, be an arbitrary element $f\in\cR$. 
Then, for any $x\in X$ and $h\in\cH$, we have
\begin{align*} f(x,h) & = \sum_{j=1}^n\alpha_j\fk_{(y_j,g_j)}(x,h) 
= \sum_{j=1}^n\alpha_j\fk((x,h),(y_j,g_j)) = \sum_{j=1}^n\alpha_j[V(x,h),V(y_j,g_j)]_{\tilde \cE} \\
& = \sum_{j=1}^n\alpha_j[\tilde V(x)h,\tilde V(y_j)g_j]_{\tilde E}= \sum_{j=1}^n\alpha_j
[h,\tilde V(x)^*\tilde V(y_j)g_j]_\cH\\
& =[h,\tilde V(x)^*\sum_{j=1}^n\alpha_j\tilde V(y_j)g_j]_\cH,
\end{align*}
hence, letting $\tilde f(x)=V(x)^*\sum_{j=1}^n\alpha_j\tilde V(y_j)g_j$, \eqref{e:fexah} holds. 
\end{remark}

\subsection{Invariant Kernels with Values Continuously Adjointable Operators.}
In this subsection we show that Theorem 2.10 in \cite{AyGheondea2} can be recovered as a 
special case of Theorem \ref{t:vhinvkolmow}. 
We first review definitions in \cite{AyGheondea2} that we will use in this subsection.

Let $X$ be a nonempty set and let $\cH$ be a VH-space over an admissible space $Z$. 
In this subsection we will consider kernels $\fk\colon X\times X\ra \cL^*_{\mathrm{c}}(\cH)$.
Such a kernel $k$ is called \emph{positive semidefinite} if it is $n$-positive for all natural numbers 
$n$, in the sense of \eqref{e:hpsd}.

 A \emph{$\cL^*_{\mathrm{c}}(\cH)$-valued VH-space linearisation} of $\fk$, or  
\emph{$\cL^*_{\mathrm{c}}(\cH)$-valued VH-space Kolmogorov decomposition} of $\fk$, is
a pair $(\cK;V)$, subject to the following conditions:
  \begin{itemize}
  \item[(vhl1)] $\cK$ is a VH-space over $Z$.
  \item[(vhl2)] $V\colon X\ra\cL^*_{\mathrm{c}}(\cH,\cK)$ satisfies $\fk(x,y)=V(x)^*V(y)$ 
for all $x,y\in X$. \end{itemize}
$(\cK;V)$ is called \emph{topologically minimal} if
  \begin{itemize}
  \item[(vhl3)] $\lin V(X)\cH$ is dense in $\cK$.
  \end{itemize}
	
We call $\fk$ \emph{$\Gamma$-invariant} if
\begin{equation}\fk(\xi\cdot x,y)=\fk(x,\xi^*\cdot y),\quad
  \xi\in\Gamma,\ x,y\in X.
\end{equation} 
A triple $(\cK;\pi;V)$ is called
a \emph{$\Gamma$-invariant $\cL^*_{\mathrm{c}}(\cH)$-valued VH-space linearisation} for $\fk$ if
\begin{itemize}
\item[(ihl1)] $(\cK;V)$ is an $\cL^*_{\mathrm{c}}(\cH)$-valued VH-space linearisation of $\fk$.
\item[(ihl2)] $\pi\colon \Gamma\ra\cL^*_{\mathrm{c}}(\cK)$ is a $*$-representation.
\item[(ihl3)] $V(\xi\cdot x)=\pi(\xi)V(x)$ for all $\xi\in\Gamma$ and all $x\in X$.
\end{itemize}
Also, $(\cK;\pi;V)$ is \emph{topologically minimal} 
if the $\cL^*_{\mathrm{c}}(\cH)$-VH-space linearisation $(\cK;V)$ is
topologically minimal, that is, $\cK$ is the closure of the linear span of $V(X)\cH$.

A VH-space $\cR$ over the ordered $*$-space 
$Z$ is called a \emph{$\cL^*_{\mathrm{c}}(\cH)$-reproducing kernel VH-space on $X$} 
if there exists a Hermitian kernel $\fk\colon X\times X\ra\cL^*_{\mathrm{c}}(\cH)$ 
such that the following axioms are satisfied:
\begin{itemize} 
\item[(rkh1)] $\cR$ is a subspace of $\cH^X$, with all algebraic operations.
\item[(rkh2)] For all $x\in X$ and all $h\in\cH$, 
the $\cH$-valued function $\fk_x h=\fk(\cdot,x)h\in\cR$.
\item[(rkh3)] For all $f\in\cR$ we have $[f(x),h]_\cH=[f,\fk_x h]_\cR$, for all 
$x\in X$ and $h\in \cH$.
\item[(rkh4)] For all $x\in X$ the evaluation operator $\cR\ni f\mapsto f(x)\in
\cH$ is continuous.
\end{itemize}

In this operator valued setting, let us note the appearance of the axiom (rkh4) which makes a
difference with classical cases, see \cite{AyGheondea2} for some results pointing out its 
significance.

\begin{theorem}[Theorem 2.10 in \cite{AyGheondea2}] \label{t:vhinvkolmo2} 
Let $\Gamma$ be a $*$-semigroup that acts 
on the nonempty set $X$ and let $\fl\colon X\times X\ra\cL_{\mathrm{c}}^*(\cH)$ 
be a kernel, 
for some VH-space $\cH$ over an admissible space $Z$. 
Then the following assertions are equivalent:

\begin{itemize}
\item[(1)] $\fl$ has the following properties:
\begin{itemize}
\item[(a)] $\fl$ is positive semidefinite, in the sense of \eqref{e:hpsd}, 
and invariant under the action of $\Gamma$ on $X$, that is, 
\eqref{e:invariant} holds.
\item[(b)] For any $\xi \in \Gamma$ and 
any seminorm $p \in S(Z)$, there exists a seminorm 
$q \in S(Z)$ and 
a constant $c_p(\xi)\geq 0$ such that for all $n \in \NN$,
$\{h_i \}_{i=1}^{n} \in \cH$, $\{ x_i \}_{i=1}^{n} \in X$ we have
\begin{equation*}
p ( \sum_{i,j=1}^{n} [ \fl(\xi\cdot x_i, \xi\cdot x_j)h_j, h_i ]_{\cH} )
\leq c_p(\xi)\,  q( \sum_{i,j=1}^{n} 
[ \fl(x_i, x_j)h_j, h_i ]_{\cH} ). 
\end{equation*} 
\item[(c)] For any $x \in X$ and any seminorm $p \in S(Z)$, 
 there exists a seminorm $q \in S(Z)$ and 
a constant $c_p(x)\geq 0$ such that for all 
$n \in \NN$, $\{ y_i \}_{i=1}^{n} \in X$, 
$\{ h_i \}_{i=1}^{n} \in \cH$ we have
\begin{equation*}
p ( \sum_{i,j=1}^{n} [ \fl(x,y_i)h_i, \fl(x, y_j)h_j ]_{\cH} )
\leq c_p(x)\,  q ( \sum_{i,j=1}^{n} 
[ \fl(y_j, y_i)h_i, h_j ]_{\cH} ). 
\end{equation*}
\end{itemize}
\item[(2)] $\fl$ has a $\Gamma$-invariant $\cL^*_{\mathrm{c}}(\cH)$-valued VH-space
linearisation $(\cK;\pi;V)$.
\item[(3)] $\fl$ admits an $\cL^*_{\mathrm{c}}(\cH)$-reproducing kernel VH-space $\cR$ and 
there exists a $*$-representation $\rho\colon \Gamma\ra\cL_{\mathrm{c}}^*(\cR)$ such that 
$\rho(\xi)\fl_xh=\fl_{\xi\cdot x}h$ for all $\xi\in\Gamma$, $x\in X$, $h\in\cH$.
\end{itemize}

In addition, in case any of the assertions \emph{(1)}, \emph{(2)}, or
\emph{(3)} holds,  
then a minimal\, $\Gamma$-invariant $\cL^*_{\mathrm{c}}(\cH)$-valued VH-space linearisation of $\fl$ can be 
constructed, and the pair $(\cR;\rho)$ as in assertion \emph{(3)} can be chosen with $\cR$ 
topologically minimal as well.
\end{theorem} 

\begin{proof} 
(1)$\Ra$(2). Define the kernel $\fk\colon (X\times\cH)\times(X\times\cH)\ra Z$ by 
\begin{equation*}
\fk((x,h),(y,g)):=[\fl(y,x)h,g]_{\cH},\quad x,y\in X,\ h,g\in\cH.
\end{equation*}
As in the proof of the implication (1)$\Ra$(2) of Theorem \ref{t:vhvergh},
$\fk$ is weakly positive semidefinite
and invariant under the action of $\Gamma$ on $X\times\cH$ given by 
$\xi\cdot(x,h)=(\xi\cdot x,h)$ for all $\xi\in\Gamma$, $x\in X$, $h\in\cH$ 
In order to see that this kernel satisfies the property 1.(c) of Theorem \ref{t:vhinvkolmo}, observe 
that for all $n\in\NN$, $\{t_i\}_{i=1}^{n}\subset\CC$, 
$\alpha\in\Gamma$, and $p\in S(Z)$, by assumption, see property 1.(b), 
there exists $q\in S(Z)$ and $c(\alpha)\geq 0$, we have   
\begin{align*}
p(\sum_{j,k=1}^n t_j \ol{t_k} \fk(\alpha\cdot(x_k,h_k),\alpha(x_j,h_j)))
&=p(\sum_{j,k=1}^n t_j\ol{t_k}[\fl(\alpha\cdot x_j,\alpha\cdot x_k)h_k,h_j]) \\
&=p(\sum_{j,k=1}^n [\fl(\alpha\cdot x_j,\alpha\cdot x_k)t_k h_k,t_j h_j]) \\
&\leq c(\alpha)^2q(\sum_{j,k=1}^n [\fl(x_j,x_k)t_k h_k,t_j h_j]) \\
&= c(\alpha)^2q(\sum_{j,k=1}^n t_j \ol{t_k} \fk((x_k,h_k),(x_j,h_j))).
\end{align*} 

By Theorem \ref{t:vhinvkolmo}, there exists a minimal weak VH-space linearisation 
$(\cK;\pi;V)$ of $\fk$ and the action of $\Gamma$ on $(X\times\cH)$. Same 
arguments as in the proof of the implication (1)$\Ra$(2) of Theorem \ref{t:vhvergh} show that, for
any $x\in X$, there exists an adjointable operator of VE-spaces 
$\tilde V(x)\colon\cH\ra\cK_0$, given by $\tilde V(x)h\colon =V(x,h)$ for $x\in X$ and $h\in\cH$, 
where $\cK_0\colon = \lin V(X)\cH$, with the property that 
$\tilde V(x)^*\tilde V(y)=\fl(x,y)$ for all $x,y\in X$. Arguing as in the 
proof of the implication (1)$\Ra$(2) of Theorem 2.10 of \cite{AyGheondea2}, it follows that
$\tilde V(x)\in\cL_{\mathrm{c}}^*(\cH,\cK_0)$. Now using the boundedness condition 
(c), for any 
$p \in S(Z)$ there exist $q\in S(Z)$ and $c_p(x)\geq 0$ such that, for all
$\sum_{i=1}^n V(y_i)h_i\in \cK_0$ we have 
\begin{align*}
 p ( [ V(x)^{*}(\sum_{i=1}^n V(y_i)h_i), V(x)^{*}(\sum_{i=1}^n V(y_i)h_i) ]_{\cH} ) & =
p ( [ \sum_{i=1}^n \fl(x,y_i)h_i, \sum_{i=1}^n \fl(x,y_i)h_i ]_{\cH} )\\
& \leq c_p(x) q ( \sum_{i,j=1}^{n} 
[ \fl(y_j, y_i)h_i, h_j ]_{\cH} ) \\
& = c_p(x)\, q ( 
[ \sum_{i=1}^n V(y_i)h_i, \sum_{i=1}^n V(y_i)h_i ]_{\cK_0} ) 
\end{align*}
hence $\tilde V(x)^*\in\cL_{\mathrm{c}}^*(\cK_0,\cH)$ for any $x\in X$. Consequently,
$\tilde V(x)^*$ extends uniquely to 
an operator $\tilde V(x)^*\in\cL^*_{\mathrm{c}}(\cK,\cH)$ for each $x\in X$. It 
follows that $(\cK;\pi;\tilde V)$ is an invariant $\cL^*_{\mathrm{c}}(\cH)$-valued VH-space 
linearisation of the kernel $\fl$ and the action of $\Gamma$ on $X$. 

(2)$\Ra$(3). Let $(\tl{\cK};\tl{\pi};\tl{V})$ be an invariant $\cL^*_{\mathrm{c}}(\cH)$-valued 
VH-space linearisation of the kernel $\fl$. In order to avoid repetition, we use some facts
obtained during the proof of the implication (2)$\Ra$(3) of Theorem~\ref{t:vevergh}.
Define $V\colon(X\times\cH)\ra\tl{\cK}$ by 
$V(x,h)=\tl{V}(x)h$ for all $x\in X$ and $h\in\cH$. Letting 
$\tl{\cK}_0=\overline{\lin V(X,\cH)}\subseteq \tl{\cK}$, similarly we see that
 $(\overline{\lin V(X,\cH)};\tl{\pi}_0;V)$ is a topological minimal invariant weak VH-space 
linearisation for the kernel $\fk\colon(X\times\cH)\times(X\times\cH)\ra Z$ 
defined by $\fk((x,h),(y,g))=[V(x,h),V(y,g)]$ for all $x,y\in X$ and $h,g\in\cH$ 
and the action of $\Gamma$ on $(X\times\cH)$ defined by 
$\xi\cdot(x,h)=(\xi\cdot x,h)$ for all $\xi\in\Gamma$, $x\in X$ and $h\in\cH$.

By Theorem \ref{t:vhinvkolmo} there exists a topologically minimal 
weak $Z$-reproducing kernel VH-space 
$\cR$ and a $*$-representation $\rho\colon\Gamma\ra\cL^*_{\mathrm{c}}(\cR)$ such that 
$\rho(\xi)\fk_{(x,h)}=\fk_{\xi\cdot(x,h)}$ 
for all $\xi\in\Gamma$, $x\in X$, $h\in\cH$. The rest of the proof is similar with the end of
the proof of the implication (2)$\Ra$(3) as in Theorem~\ref{t:vevergh}. We show that,
for each $f\in\cR$ and $x\in X$ there exists a unique element $\tl{f}(x)$ such that \eqref{e:fexah} 
holds and, consequently, this gives rise to a map $\cR\ni f\mapsto Uf=\tl{f}\in \cH^X$, which is
linear and bijective between $\cR$ and its range $\tl{\cR}\subseteq \cH^X$. 
Letting $[\tl{f},\tl{g}]_{\tl{\cR}}=[f,g]_\cR$ for all $f,g\in\cR$, $\tl{\cR}$ becomes an $\cH$-valued
reproducing kernel VH-space with kernel $\fl$, and then letting
$\tl{\rho}\colon = U\rho U^*$, $(\tl{\cR},\tl{\rho})$ is a pair having all the required properties.

(3)$\Ra$(1). Assume that the pair $(\tl{\cR};\tl{\rho})$ consists of an
$\cL^*_{\mathrm{c}}(\cH)$-valued reproducing kernel VH-space of $\fl$ and a $*$-representation 
$\tl{\rho}$ of $\Gamma$ on $\cL_{\mathrm{c}}^*(\tl{\cR})$ such that 
$\rho(\xi)\fl_{x}h=\fl_{\xi\cdot x}h$ for all $\xi\in\Gamma$, 
$x\in X$, $h\in\cH$. 
Similarly as in the proof of the implication (3)$\Ra$(1) 
of Theorem \ref{t:vevergh}, the kernel $\fl$ is 
shown to be positive semidefinite and invariant under 
the action of $\Gamma$ on $X$. 
On the other hand, the inequalities (b) and (c) are obtained from the continuity of the operator
$\rho(\xi)\colon\cR\ra\cR$, for any $\xi\in\Gamma$ and, respectively, from the continuity of the
evaluation operator $E_x\colon\cR\ra\cH$, for any $x\in X$. 
\end{proof}

\subsection{Invariant Kernels with Values Boundedly Adjointable Operators.}
We show that Theorem 4.2 in \cite{Gheondea} is a special case of Theorem \ref{t:vhinvkolmo}. 
We review necessary definitions in \cite{Gheondea}. 

Given a $\cB^*(\cH)$-valued kernel $\fl$ on a nonempty set $X$, where 
$\cH$ is a VH-space over the admissible space $Z$, a \emph{$\cB^*(\cH)$-valued VH-space 
linearisation} of $\fl$ is a pair $(\tl{\cK};\tl{V})$ with 
\begin{itemize}
\item[(hvhl1)] $\tl{\cK}$ is a VH-space over $Z$. 
\item[(hvhl2)] $\tl{V}\colon X\ra\cB^*(\cH;\cK)$ satisfies 
$\fl(x,y)=\tl{V}(x)^*\tl{V}(y)$ for all $x,y\in X$.
\end{itemize}

If $\Gamma$ is a $*$-semigroup acting on $X$, $(\tl{\cK};\tl{\pi};\tl{V})$ is 
called an \emph{invariant $\cB^*(\cH)$-valued VH-space linearisation} 
of the kernel $\fl$ and the action 
of $\Gamma$ on $X$, if, in addition to (hvhl1) and (hvhl2), we have,
\begin{itemize}
\item[(hvhl3)] $\tl{\pi}\colon \Gamma \ra \cB^*(\tl{\cK})$ is a $*$-representation.
\item[(hvhl4)] $\tl{V}(\xi\cdot x)=\tl{\pi}(\xi)V(x)$ for every $\xi\in\Gamma$, $x\in X$.
\end{itemize}
If we have 
\begin{itemize}
\item[(hvhl5)] $\mathrm{Lin} \tl{V}(X)\cH$ is dense in $\tl{\cK}$,
\end{itemize}
then $(\tl{\cK};\tl{\pi};\tl{V})$ is called \emph{topologically minimal}.

Given
a nonempty set $X$ and a VH-space $\cH$ over the 
admissible space $Z$, a VH-space $\tl{\cR}$ over $Z$ is 
called a \emph{$\cB^*(\cH)$-valued reproducing kernel VH-space} on $X$ if there exists a 
kernel $\fl\colon X\times X\ra \cB^*(\cH)$ such that 
\begin{itemize}
\item[(hrk1)] $\tl{\cR}$ is a subspace of $\cH^X$ with all 
algebraic operations.
\item[(hrk2)] $\fl_{x}h=\fl(\cdot,x)h\in\tl{\cR}$ for all $x\in X$, $h\in\cH$. 
\item[(hrk3)] $[f(x),h]_{\cH}=[f,\fl_xh]_{\tl{\cR}}$ holds for all $f\in\tl{\cR}$, $x\in X$ and $h\in\cH$. 
\end{itemize}
$\tilde\cR$ is  called \emph{topologically minimal} if
\begin{itemize}
\item[(hrk4)] $\mathrm{Lin} \{ \fl_xh \mid x\in X, h\in \cH \}$ is dense in $\tl{\cR}$,
\end{itemize}
and, in this case, $\tilde\cR$ is uniquely determined by the kernel $\fl$. 

\begin{theorem}[Theorem 4.2 in \cite{Gheondea}] \label{t:vhvergh}%ve version of Gh,2012
Let $\Gamma$ be a $*$-semigroup acting on a nonempty set $X$, 
$\cH$ be a VH-space on an 
admissible space $Z$, 
and $\fl:X \times X \ra \cB^*(\cH)$ be
a kernel. Then the following are equivalent:
\begin{itemize}
\item[(1)] $\fl$ has the following properties:
\begin{itemize}
\item[(a)] $\fl$ is positive semidefinite.
\item[(b)] $\fl$ is invariant under the action of $\Gamma$ on $X$. 
\item[(c)] For any $\alpha \in \Gamma$ there exists
$c(\alpha) \geq 0$ such that, for all $n \in \NN$, $x_1,x_2,\cdots x_n \in X$,
$h_1,h_2,\cdots,h_n \in \cH$, we  have
\begin{equation} \label{e:tecbc}%technical bddness condition
\sum_{i,j=1}^{n} [ \fl(\alpha\cdot x_i,\alpha\cdot x_j)h_j,h_i ]_{\cH}
\leq c(\alpha)^{2}[\fl(x_i,x_j)h_j,h_i]_{\cH}.
\end{equation}

\end{itemize}
\item[(2)] $\fl$ has an invariant $\cB^*(\cH)$-valued VH-space linearisation 
$(\tl{\cE};\tl{\pi};\tl{V})$.
\item[(3)] $\fl$ admits a $\cB^*(\cH)$-reproducing kernel 
VH-space $\tl{\cR}$ and there exists a
$*$-representation $\tl{\rho}: \Gamma \ra \cB^*(\tl{\cR})$ such that
$\tl{\rho}(\xi) \fl_x h = \fl_{\xi\cdot x}h$ 
for all $\xi \in \Gamma$, $x \in X$, $h \in \cH$.
\end{itemize}
Moreover, in case any of the assertions (1), (2) or (3)
holds, a topologically minimal invariant $\cB^*(\cH)$-valued VH-space linearisation 
can be constructed.
\end{theorem}

\begin{proof}
(1)$\Ra$(2). Define the kernel $\fk\colon (X\times\cH)\times(X\times\cH)\ra Z$ by 
\begin{equation*}
\fk((x,h),(y,g)):=[\fl(y,x)h,g]_{\cH}
\end{equation*}
for all $x,y\in X$ and $h,g\in\cH$. Then $\fk$ is weakly positive semidefinite 
and invariant under the action of $\Gamma$ on $X\times\cH$ given by 
$\xi\cdot(x,h)=(\xi\cdot x,h)$ for all $\xi\in\Gamma$, $x\in X$, $h\in\cH$, 
as in the proof of (1)$\Ra$(2) of Theorem \ref{t:vevergh}. To see that this kernel 
satisfies condition 1.(c) of Theorem \ref{t:vhinvkolmo}, we use the assumption (c) and get
that, for any $\alpha\in\Gamma$ there exists $c(\alpha)\geq 0$ such that, 
for all $n\in\NN$, $\{t_i\}_{i=1}^{n}\subset\CC$, we have
\begin{align*}
\sum_{j,k=1}^n t_j \ol{t_k} \fk(\alpha\cdot(x_k,h_k),\alpha(x_j,h_j))
&=\sum_{j,k=1}^n t_j\ol{t_k}[\fl(\alpha\cdot x_j,\alpha\cdot x_k)h_k,h_j] \\
&=\sum_{j,k=1}^n [\fl(\alpha\cdot x_j,\alpha\cdot x_k)t_k h_k,t_j h_j] \\
&\leq c(\alpha)^2\sum_{j,k=1}^n [\fl(x_j,x_k)t_k h_k,t_j h_j] \\
&= c(\alpha)^2\sum_{j,k=1}^n t_j \ol{t_k} \fk((x_k,h_k),(x_j,h_j)).
\end{align*}

By Theorem \ref{t:vhinvkolmo}, there exists a minimal weak VH-space linearisation 
$(\cK;V;\pi)$ of $\fk$ and the action of $\Gamma$ on $(X\times\cH)$. Same 
arguments as in proof of (1)$\Ra$(2) of Theorem \ref{t:vevergh} gives 
an adjointable operator of VE-spaces 
$\tilde V(x)\colon\cH\ra\cK_0$, given by $\tilde V(x)h\colon =V(x,h)$ for $x\in X$ and $h\in\cH$, 
where $\cK_0\colon = \mathrm{Lin}V(X)\cH$, with the property that 
$\tilde V(x)^*\tilde V(y)=\fl(x,y)$ for all $x,y\in X$. Arguing as in the proof 
of Theorem 3.3 of \cite{Gheondea}, it follows that 
$\tilde V(x)\in\cB^*(\cH,\cK_0)$ and $\tilde V(x)^*\in\cB^*(\cK_0,\cH)$. Hence 
$\tilde V(x)^*$ extends uniquely to 
an operator $V(x)^*\in\cB^*(\cK,\cH)$ for each $x\in X$. It 
follows that $(\cK;\pi;\tilde V)$ is a topologically minimal invariant $\cB^*(\cH)$-valued VH-space 
linearisation of the kernel $\fl$ and the action of $\Gamma$ on $X$. 

(2)$\Ra$(3). Let $(\tl{\cK};\tl{\pi};\tl{V})$ be an invariant $\cB^*(\cH)$ 
VH-space linearisation of the kernel $\fl$. We essentially use Theorem~\ref{t:vhinvkolmo}
with details very close to the proof of the implication (2)$\Ra$(3) of Theorem~\ref{t:vhinvkolmo2},
with the difference that we obtain bounded adjointable operators instead of continuously 
adjointable operators.
Define $V\colon(X\times\cH)\ra\tl{\cK}$ by 
$V(x,h)=\tl{V}(x)h$ for all $x\in X$ and $h\in\cH$. We also 
have $\tl{\pi}(\xi)\tl{V}(x)h=\tl{V}(\xi\cdot x)h=V(\xi\cdot x,h)$ 
for all $\xi\in\Gamma$, $x\in X$, $h\in\cH$. Then 
$(\ol{\lin(V(X,\cH)};V;\tl{\pi})$ is a topologically minimal weak invariant VH-space 
linearisation for the kernel $\fk\colon(X\times\cH)\times(X\times\cH)\ra Z$ 
defined by $\fk((x,h),(y,g))=[V(x,h),V(y,g)]$, for all $x,y\in X$ and $h,g\in\cH$,
and the action of $\Gamma$ on $(X\times\cH)$ defined by 
$\xi\cdot(x,h)=(\xi\cdot x,h)$, for all $\xi\in\Gamma$, $x\in X$, and $h\in\cH$. 

By Theorem \ref{t:vhinvkolmo} there exists a weak $Z$-reproducing kernel VH-space 
$\cR$ and a $*$-representation $\rho\colon\Gamma \ra\cB^*(\cR)$ such that 
$\rho(\xi)\fk_{(x,h)}=\fk_{\xi\cdot(x,h)}$ 
for all $\xi\in\Gamma$, $x\in X$, $h\in\cH$. Define $\tl{f}\colon X\ra\cH$ 
as follows: for each $x\in X$ let $\tl{f}(x)\in\cH$ be the unique 
element satisfying 
$[\tl{f}(x),h]_{\cH}=f(x,h)$ for all $h\in\cH$ and let 
$\tl{\cR}$ be the vector space of all $\tl{f}$, when $f\in\cR$. Since we have 
\begin{equation*}
[\fk_{(x,h)},\fk_{(y,g)}]=\fk_{(x,h)}(y,g)=[\widetilde{\fk_{(x,h)}}(y),g],\quad x,y\in X,\ h,g\in\cH,
\end{equation*}
it follows that $\fl_{x}h=\widetilde{\fk_{(x,h)}}\in\tl{\cR}$ for all $x\in X$ and $h\in\cH$. 

It is easy to check that the map 
$U\colon\cR\ni f \ra \tl{f}\in\tl{\cR}$ is linear, one-to-one, and onto. 
Therefore, defining $[\tl{f},\tl{g}]_{\tl{\cR}}:= [f,g]_{\cR}$ 
makes $\tl{\cR}$ a $\cB^*(\cH)$-reproducing kernel VH-space 
with reproducing kernel $\fl$, and $U$ becomes a unitary 
operator of VH-spaces. Defining $\tl{\rho}:= U\rho U^*$, the pair
$(\tl{\cR},\tl{\rho})$ has all the required properties.

(3)$\Ra$(1). Assume that $(\tl{\cR};\tl{\rho})$ is a 
$\cB^*(\cH)$-reproducing kernel VH-space of $\fl$ with a representation 
$\tl{\rho}$ of $\Gamma$ on $\cB^*(\tl{\cR})$ such that 
$\rho(\xi)\fl_{x}h=\fl_{\xi\cdot x}h$ for all $\xi\in\Gamma$, 
$x\in X$, $h\in\cH$. Similarly as in proof 
of the implication (3)$\Ra$(1) of Theorem \ref{t:vevergh}, the kernel $\fl$ is 
shown to be positive semidefinite and invariant under 
the action of $\Gamma$ on $X$. On the other hand, using the fact that the linear operator
$\tilde\rho(\xi)\colon\cH\ra\cH$ is bounded for all $\xi\in\Gamma$, it follows that, for any 
$\xi\in\Gamma$, there exists 
$c(\xi)\geq 0$ such that, for all $\xi\in\Gamma$, $n\in\NN$, $\{ x_i \}_{i=1}^n\in X$ and 
$\{ h_i \}_{i=1}^n\in \cH$, we have
\begin{align*}
\sum_{j,k=1}^n [\tl{\fl}(\xi\cdot x_j,\xi\cdot x_k)h_k,h_j]_{\cH}
&=\sum_{j,k=1}^n [\tl{\fl}_{\xi\cdot x_k}h_k(\xi\cdot x_j),h_j]_{\cH} 
=\sum_{j,k=1}^n [\tl{\fl}_{\xi\cdot x_k}h_k, \tl{\fl}_{\xi\cdot x_j}h_j]_{\tl{\cR}} \\
&=[\tl{\rho}(\sum_{k=1}^n\tl{\fl}_{x_k}h_k), \tl{\rho}(\sum_{j=1}^n\tl{\fl}_{x_j}h_j)]_{\tl{\cR}} 
\leq c(\xi)^2 [\sum_{k=1}^n\tl{\fl}_{x_k}h_k, \sum_{j=1}^n\tl{\fl}_{x_j}h_j]_{\tl{\cR}} \\
&=c(\xi)^2 \sum_{j,k=1}^n [\tl{\fl}(x_j, x_k)h_k,h_j]_{\cH},
\end{align*}
hence $\fl$ has the property (c).
\end{proof}

\subsection{Positive Semidefinite $\cL(\cX,\cX_Z^\prime)$ Valued Maps on $*$-Semigroups.}
In this subsection we obtain stronger versions of Theorem~3.1 and Theorem~4.2 in
\cite{PaterBinzar} as
applications of Theorem~\ref{t:veinvkolmo} and, respectively, of Theorem~\ref{t:vhinvkolmo}.
We first reorganise some definitions from \cite{PaterBinzar} and \cite{GorniakWeron}. 

Let $\cX$ be a vector space, and $Z$ be an ordered  $*$-space. 
By $\cX_Z^{\prime}$ we denote the space of all conjugate linear functions 
from $\cX$ to $Z$  and call it the \emph{algebraic conjugate $Z$-dual space}. Let 
$\cL(\cX,\cX_Z^\prime)$ denote the vector space of all linear operators 
$T\colon \cX\ra \cX_Z^\prime$.
For any VE-space $\cE$ over $Z$ and any linear operator $A\colon \cX\ra \cE$, we define a linear 
operator $A^\prime\colon \cE\ra \cX_Z^\prime$, called the \emph{algebraic $Z$-adjoint operator},
by
\begin{equation}\label{e:zaoa} (A^\prime f)(x)=[Ax, f]_\cE,\quad f\in\cE,\ x\in \cX.
\end{equation}

If $\Gamma$ is a $*$-semigroup, a map $T\colon \Gamma\ra\cL(\cX,\cX_Z^\prime)$ is called
\emph{$\cL(\cX,\cX_Z^\prime)$-valued $n$-positive} if 
\begin{equation}\label{e:psd2}\sum_{i,j=1}(T_{s_i^*s_j}x_j)(x_i)\geq 0_Z
\mbox{ for all }
(s_i)_{i=1}^n\in\Gamma\mbox{ and all } (x_j)_{j=1}^n\in \cX.\end{equation} 
If $T$ is $n$-positive for all $n\in\NN$ then it is called \emph{$\cL(\cX,\cX_Z^\prime)$-valued 
positive semidefinite}.

\begin{remarks}\label{r:wpsd} With notation as before, let 
$T\colon \Gamma\ra\cL(\cX,\cX_Z^\prime)$.

(1) We define a kernel $\fk\colon (\Gamma\times \cX)\times(\Gamma\times \cX) \ra Z$ by
\begin{equation}\label{e:fekasextey}
\fk((s,x),(t,y))=(T_{s^*t}y)x,\quad s,t\in\Gamma,\ x,y\in\cX.
\end{equation}
Then for all $n\in\NN$, all
$\alpha_1,\alpha_2,\dots,\alpha_n\in\CC$, and all $(s_i,x_i)_{i=1}^n\in(S\times X)$ we have
\begin{align*}
\sum_{i,j=1}^n\overline\alpha_i\alpha_j\fk((s_i,x_i),(s_j,x_j))
=\sum_{i,j=1}\overline\alpha_i\alpha_j(T_{s_i^*s_j}x_j)x_i 
=\sum_{i,j=1}(T_{s_i^*s_j}\alpha_jx_j)(\alpha_ix_i).
\end{align*}
This shows that, for $n\in\NN$, the map $T$ is $n$-positive if and only if the kernel $\fk$ is weakly 
$n$-positive. In particular, $T$ is positive semidefinite if and only if the kernel $\fk$ is weakly
positive semidefinite.

(2) Recall that, see \eqref{e:herm}, the kernel $\fk$ is Hermitian if 
$\fk((s,x),(t,y))=\fk((t,y),(s,x))^*$ for all $s,t\in\Gamma$ and all $x,y\in\cX$. 
From \eqref{e:fekasextey} it follows that $\fk$ is Hermitian if and only if
\begin{equation}\label{e:tesety}
(T_{s^*t}y)x=((T_{t^*s}x)y)^*,\quad\mbox{ for all }s,t\in\Gamma,\mbox{ and all }x,y\in\cX.
\end{equation}
Consequently, by Lemma~\ref{l:twopos} it follows that, if $T$ is $2$-positive, then \eqref{e:tesety} 
holds.

In addition, if $\Gamma$ has a unit $e=e^*$, then \eqref{e:fekasextey} is equivalent with
\begin{equation}\label{e:tesy}
(T_{s^*}y)x=((T_{s}x)y)^*,\quad\mbox{ for all }s\in\Gamma,\mbox{ and all }x,y\in\cX.
\end{equation}

(3) We define a left action of $\Gamma$ 
on $(\Gamma\times \cX)$ by 
\begin{equation}\label{e:usex}
u\cdot(s,x)=(us,x),\mbox{ for all }u,s\in\Gamma,\mbox{ and all } 
x\in \cX.\end{equation}
 For all $u\in\Gamma$ and all $(s,x)\in \Gamma\times \cX$
we have 
\begin{equation*}
\fk((s,x),u\cdot(t,y))=(T_{s^*ut}y)x 
=(T_{(u^*s)^*ty})x=\fk(u^*(s,x),(t,y)),
\end{equation*}
hence the kernel $\fk$ is invariant under the left action of $\Gamma$ on $\Gamma\times\cX$
defined as in \eqref{e:usex}.
\end{remarks}

\begin{theorem}\label{t:pateralg} Let $Z$ be an ordered $*$-space, let
$\cX$ be complex vector space with algebraic conjugate $Z$-dual space $\cX_Z^\prime$, 
and consider $T\colon \Gamma\ra \cL(\cX,\cX_Z^\prime)$, for some $*$-semigroup $\Gamma$ with 
unit. The following assertions are equivalent:
\begin{itemize}
\item[(i)] $T$ is positive semidefinite, in the sense of \eqref{e:psd2}.
\item[(ii)] There exist a VE-space $\cE$ over $Z$, a unital 
$*$-representation $\pi\colon\Gamma\ra\cL^*(\cE)$, and an operator $A\in\cL(\cX,\cE)$, such that
\begin{equation}\label{e:tegad} T_t=A^\prime \pi(t) A,\quad t\in\Gamma.
\end{equation}
\end{itemize}
If any of the conditions (i) and (ii) holds, then the VE-space $\cE$ 
can be chosen minimal in the sense that it coincides with the linear span of 
$\pi(\Gamma) A\cX$ and, in this case, it is unique modulo a unitary
equivalence.
\end{theorem}

\begin{proof} (i)$\Ra$(ii). We consider the kernel $\fk\colon\Gamma\times\cX\ra Z$ as in 
\eqref{e:fekasextey} and the left action of $\Gamma$ on $\Gamma\times\cX$ as in \eqref{e:usex}.
By Remark~\ref{r:wpsd}.(1) and Remark~\ref{r:wpsd}.(2), $\fk$ is a $Z$-valued weakly positive
semidefinite kernel invariant under the action of $\Gamma$ as in \eqref{e:usex}
hence, by Theorem~\ref{t:veinvkolmo}, there exists a minimal invariant weak
VE-space linearisation $(\cE,\pi,V)$ of $\fk$. Since
\begin{equation*}
[V(s,x),V(t,y)]_\cE=\fk((s,x),(t,y))=(T_{s^*t}y)x,\quad s,t\in\Gamma,\ x,y\in\cX,
\end{equation*}
it follows that $\cX\ni x\mapsto V(s,x)\in \cE$ is linear, for all $s\in\Gamma$. This shows that,
we can define $\tilde V\colon\Gamma\ra\cL(\cX,\cE)$ by $\tilde V(s)x=V(s,x)$, for all $s\in\Gamma$
and all $x\in\cX$. Taking into account \eqref{e:zaoa} it follows that
\begin{equation*}
(\tilde V(s)^\prime f)x=[\tilde V(s)x,f]_\cE,\quad s\in \Gamma,\ x\in\cX,\ f\in\cE,
\end{equation*}
hence, letting $A=\tilde V(e)\in\cL(\cX,\cE)$ it follows that, for all $s\in\Gamma$ and all $x,y\in\cX$
we have
\begin{align*}
(A^\prime \pi(s)Ax)y & =(\tilde V(e)^\prime \pi(s)\tilde V(e)x)y = [\tilde V(e)y,\pi(s)\tilde V(e)x]_\cE \\
& = [V(e,y),\pi(s)V(e,x)]_\cE= [V(e,y),V(s,x)]_\cE=  
\fk((e,y),(s,x))=(T_sy)x,
\end{align*} and hence \eqref{e:tegad} is proven. The minimality and the uniqueness property
follow by standard arguments that we omit.

(ii)$\Ra$(i). This follows by a standard argument that we omit.
\end{proof}

Theorem~\ref{t:pateralg} is stronger than Theorem~3.1 
in \cite{PaterBinzar} since, in addition to positive semidefiniteness of $T$ they require the 
condition \eqref{e:tesy} as well. As we have seen in Remark~\ref{r:wpsd}.(3), this condition
is a consequence of the positive semidefiniteness of $T$. Also, the ordered $*$-space $Z$ need 
not be admissible, actually, the topology of $Z$ does not play any role.

From now on we assume that $Z$ is a topologically ordered $*$-space and that
$\cX$ is a locally bounded topological vector space, that is, in $\cX$ 
there exists a bounded neighbourhood of $0$. By
$\cX_Z^*$ we denote the subspace of $\cX_Z^\prime$ of all continuous conjugate linear functions 
from $\cX$ to $Z$ and call it the \emph{topological conjugate $Z$-dual space}. The 
space $\cX_Z^*$ is considered 
with the topology of uniform convergence on bounded sets, that is, 
a net $(f_i)_{i\in\cI}\in \cX_Z^*$ 
converges to $0$ if for any bounded subset $B\subset\cX$ 
the $Z$-valued net $(f_i(y))_{i\in\cI}$ converges 
to $0$ uniformly with respect to $y\in B$, equivalently, for any bounded set $B\subset\cX$,
any $p\in S(Z)$ 
and any $\epsilon > 0$, there exists $i_0\in \cI$ such that 
$i \geq i_0$ implies $p(f_i(y))<\epsilon$ for all $y\in B$. 
Let $\cL_{\mathrm{c}}(\cX,\cX_Z^*)$ be the space 
of all continuous linear operators from $\cX$ to $\cX_Z^*$. 

Let $\cE$ be a VE-space over $Z$, with topology defined as in Subsection~\ref{ss:vhs}.
Following \cite{PaterBinzar} and \cite{GorniakWeron2}, for any $A\in\cL_{\mathrm{c}}(\cX,\cE)$
the \emph{topological $Z$-adjoint operator} of $A$ is, by definition, the operator 
$A^*\colon \cE\ra \cX_Z^*$ defined by
\begin{equation}\label{e:zao} (A^* f)x=[Ax,f]_\cE,\quad f\in\cE,\ x\in \cX.
\end{equation}
By Lemma~\ref{l:schwarz} the definition of $A^*$ is correct.

\begin{theorem}\label{t:pater}
Let $\Gamma$ be a $*$-semigroup with unit $e$ and $\cX$ be a locally bounded
topological vector space 
with topological conjugate $Z$-dual space $\cX_Z^*$ for an admissible space $Z$. Let 
$T\colon\Gamma\ra\cL_{\mathrm{c}}(\cX,\cX_Z^\prime)$ subject to the following properties:
\begin{itemize}
\item[(a)] $T$ is
 an $\cL(\cX,\cX_Z^\prime)$-valued positive semidefinite map.
 \item[(b)] For all $u\in\Gamma$, there is a constant $c(u)\geq 0$ such that 
for all $n\in\NN$, all $s_1,\ldots,s_n\in\Gamma$, and all $x_1,\ldots,x_n\in\cX$,
we have
\begin{equation}\label{e:sznagy2}\sum_{i,j=1}^n(T_{s_i^*u^*us_j}x_j)(x_i) 
\leq c(u)^2\sum_{i,j=1}^n (T_{s_i^*s_j}x_j)(x_i).\end{equation}
 \item[(c)] $T(e)\in\cL_{\mathrm{c}}(\cX,\cX_Z^*)$.
 \end{itemize}
 
Then:
\begin{itemize}
\item[(i)] There exist a VH-space $\cK$ over $Z$,
a $*$-representation $\pi\colon\Gamma\ra\cB^*(\cK)$ 
and an operator $A\in\cL_{\mathrm{c}}(\cX,\cK)$, such that 
$T_s=A^*\pi(s)A$ for any $s\in\Gamma$.
\item[(ii)] $T_s\in\cL_{\mathrm{c}}(\cX,\cX_Z^*)$ for all $s\in\Gamma$.
\item[(iii)] If $(u_l)_{l\in \cL}$ is a net in $\Gamma$ with $\sup_{l\in \cL}c(u_l)<\infty$ 
and $(T_{su_lt})_{l\in \cL}$ converges to $T_{sut}$, for some $u\in\Gamma$ and 
any $s,t\in\Gamma$, in the 
weak topology of $\cL_{\mathrm{c}}(\cX,\cX_Z^*)$, then $(\pi(u_l))_{l\in \cL}$ converges to 
$\pi(u)$ in the weak topology of $\cB^*(\cK)$. 
\end{itemize} 
\end{theorem}

\begin{proof}
Define the kernel $\fk\colon (\Gamma\times \cX)\times(\Gamma\times \cX) \ra Z$ as in 
\eqref{e:fekasextey}. By Remark~\ref{r:wpsd}, it follows that $\fk$ is a $Z$-valued weakly positive
semidefinite kernel.
Next, consider the left action of $\Gamma$ 
on $(\Gamma\times \cX)$ as in \eqref{e:usex} and by Remark~\ref{r:wpsd} it follows that
$\fk$ is invariant under this action.
 In order to show that the property 1.(c) of Theorem \ref{t:vhinvkolmo} holds, 
let $u\in\Gamma$, $n\in\NN$, $(s_i,x_i)_{i=1}^n\in(\Gamma\times \cX)$. Then, using
\eqref{e:sznagy2} it follows that
\begin{align*}
\sum_{i,j=1}^n\fk(u\cdot(s_i,x_i),u\cdot(s_j,x_j))&=\sum_{i,j=1}^n(T_{s_i^*u^*us_j}x_j)x_i \\
&\leq c(u)^2\sum_{i,j=1}^n (T_{s_i^*s_j}x_j)(x_i)
=c(u)^2\sum_{i,j=1}^n \fk((s_i,x_i),(s_j,x_j)).
\end{align*}

By Theorem \ref{t:vhinvkolmo}, there exists a topologically minimal 
invariant weak VH-space linearisation 
$(\cK;\pi;V)$ of the kernel $\fk$. Since 
$[V(s,x),V(t,y)]_{\cK}=\fk((s,x),(t,y))=(T_{s^*t}y)(x)$ for all 
$s,t\in \Gamma$ and $x,y\in \cX$, we observe that $V(s,x)$ depends linearly 
on $x\in \cX$ for each $s\in\Gamma$. As a consequence, letting $\tl{V}(s)x=V(s,x)$, for all 
$x\in\cX$, we obtain a linear operator 
$\tl{V}(s):\cX\ra\cK$ for each $s\in\Gamma$. 
To see that $\tl{V}(s)$ is continuous for each $s\in \Gamma$, let 
$(x_l)_{l\in\cL}$ be a net in $\cX$ converging to $0$. Since $\cX$ is locally bounded, there
exists $B\subset\cX$ a bounded neighbourhood of $0$ and then there exists $l_1\in\cL$ such 
that $(x_l)_{l\geq l_1}$ is contained in $B$. Since $T_e\in\cL_{\mathrm{c}}(\cX,\cX_Z^*)$, 
taking into account the topology of $\cX_Z^*$,
given any $\varepsilon > 0$ and any $p\in S(Z)$
we can find $l_2\in\cL$ such that $\cL\ni l\geq l_2$ implies $p((T_ex_l)y)<\varepsilon$ 
for all $y\in B$. Since $\cL$ is directed, there exists $l_0\in\cL$ with $l_0\geq l_1$ and 
$l_0\geq l_2$.
Then, for any $l\geq l_0$, by  \eqref{e:sznagy2} and taking into account how 
the topology of $\cK$ is defined, see Subsection~\ref{ss:vhs},
we have
\begin{align*}
p([\tl{V}(s)x_l,\tl{V}(s)x_l]_{\cK})&=p(\fk((s,x_l),(s,x_l)) \\
&=p(\fk(s\cdot(e,x_l),s\cdot(e,x_l)) 
\leq c(s)^2 p(\fk((e,x_l),(e,x_l))) \\
&=c(s)^2p((T_ex_l)x_l) \leq c(s)^2\sup_{y\in B} p((T_ex_l)y) \leq c(s)^2\varepsilon,
\end{align*} 
 hence $\tl{V}(s)\in\cL_{\mathrm{c}}(\cX,\cK)$, for any $s\in\Gamma$. 
In addition, for each $s\in\Gamma$ the operator $\tl{V}(s)^*\in\cL(\cK,\cX_Z^*)$ is defined as in \eqref{e:zao}.

Letting $A:=\tl{V}(e)$ we have
\begin{align*}
(A^*\pi(s)Ax)y&=(\tl{V}(e)^*\pi(s)\tl{V}(e)x)(y)=[V(e,y),V(s,x)]_{\cK} \\
&=\fk((e,y),(s,x))=(T_sx)y\,\,\mbox{for all}\,\,s\in\Gamma\,\,\mbox{and}\,\,x,y\in \cX.
\end{align*}
Therefore $A^*\pi(s)A=T_s\in\cL_{\mathrm{c}}(\cX,\cX_Z^*)$, for all $s\in\Gamma$.

The rest of the proof, which shows that $\pi(u_l)_{l\in \cL}$ 
converges to $\pi(u)$ in the weak topology of $\cB^*(\cK)$, as 
in the second part of the conclusion, uses standard arguments and 
is the same with that in \cite{PaterBinzar}. 
For completeness, 
we present it here. Let $\cK_0:=\mathrm{Lin}V(\Gamma\times \cX)$. 
By minimality, $\cK_0$ is dense in $\cK$. Let $e,f\in\cK_0$, 
with $e=\sum_{i=1}^n\alpha_iV(s_i,x_i)$ and $f=\sum_{j=1}^m\beta_jV(r_j,y_j)$. We have
\begin{align*}
[e,\pi(u_l)f]_{\cK}& =[(\sum_{i=1}^n\alpha_iV(s_i,x_i)),\pi(u_l)\sum_{j=1}^m\beta_jV(r_j,y_j)]_{\cK} \\
& =\sum_{i=1}^n\sum_{j=1}^m\overline\alpha_i\beta_j(T_{s_i^*ur_j}y_j)(x_i)\xrightarrow[l]{ }
\sum_{i=1}^n\sum_{j=1}^m\overline\alpha_i\beta_j(T_{s_i^*ur_j}y_j)x_i=[e,\pi(u)f]_{\cK}
\end{align*} 
by the assumption that $(T_{su_lt})_{l\in \cL}$ converges to $T_{sut}$ for any $s,t\in\Gamma$. 
Now let $g,h\in\cK$ and let $(g_i)_{i\in \cI}\in\cK_0$ be a net converging to $g$. 
For any $p\in S(Z)$, $j\in \cI$ and $l\in \cL$ we have
\begin{align*}
p([(\pi(u)-\pi(u_l))g,h]_{\cK})& \leq p([\pi(u)(g-g_j),h]_{\cK})+p([(\pi(u)-\pi(u_l))g_j,h]_{\cK}) \\
&\phantom{p([\pi(u)(g-g_j),h]_{\cK})+p([(\pi(u)}
+p([\pi(u_l)(g_j-g),h]_{\cK})\\
& \leq  4\tl{p}(h)\bigl(\tl{p}(\pi(u)(g-g_j))+\tl{p}(\pi(u_l)(g_j-g))+\tl{p}((\pi(u)-\pi(u_l))g_j)\bigr) \\
& \leq  4\tl{p}(h)\bigl(c^2\tl{p}(g-g_j)+\tl{p}((\pi(u)-\pi(u_l))g_j)\bigr)
\end{align*}
for some constant $c$, where the second inequality follows by the Schwarz type 
inequality \eqref{e:schwarz} and the third inequality by the fact that $\pi(u),\pi(u_l)\in \cB^*(\cK)$ 
and the assumption that $\sup_{l\in \cL}c(u_l)<\infty$. 
Now that weak convergence was shown in $\cK_0$, 
a standard argument finishes the proof.
\end{proof}

\begin{remarks}\label{r:rf} 
(1) Theorem~\ref{t:pater} is stronger than Theorem 4.2 of \cite{PaterBinzar}, see also the correction 
in \cite{PaterBinzar2}, with respect to two aspects: firstly, since they have the additional
assumption that \eqref{e:tesy} holds,
which is actually a consequence of positive semidefiniteness, as Remark~\ref{r:wpsd}.(2) shows,
and secondly since their assumption $T_s\in\cL_{\mathrm{c}}(\cX,\cX_Z^*)$, for all $s\in\Gamma$, 
is actually a consequence of the weaker one $T_e\in\cL_{\mathrm{c}}(\cX,\cX_Z^*)$, as the proof
of Theorem~\ref{t:pater} shows.

(2) It is easy to see that, 
there is a "converse" to Theorem~\ref{t:pater} in the sense that, if assertion (i)
is assumed, then assertions (a), (b), (c),  and (ii) are obtained as consequences.
\end{remarks}

\end{document}